%% file: exploding_rand.tex
\documentclass[12pt]{amsart}
\usepackage{fullpage,booktabs, graphicx,psfrag,amsmath, amssymb,amsthm, hyperref}

\input defs.tex

\title{A shape theorem for exploding sandpiles}
\author{Ahmed Bou-Rabee}

\begin{document}

	\begin{abstract}
		We study scaling limits of exploding Abelian sandpiles using ideas from percolation and front propagation in random media. 
		We establish sufficient conditions under which a limit shape exists and show via a family of counterexamples that convergence may not occur in general. A corollary of our proof is a simple criteria for determining if a sandpile is explosive; this 
		strengthens a result of Fey, Levine, and Peres (2010).

	\end{abstract}
\maketitle

\section{Introduction}
\subsection{Overview}
We consider the Abelian sandpile growth model on the integer lattice. Start with a {\it background} of indistinguishable chips, $\eta: \Z^d \to \Z$, add $n$ chips at the origin, 
and attempt to {\it stabilize} via {\it parallel toppling}:
\begin{equation} \label{eq:parallel_toppling}
\begin{aligned}
v_{t+1} &= v_{t} + 1 \{ s_{t} \geq 2 d \} \\
s_{t+1} &= s_{t} + \Delta (v_{t+1} - v_{t}), 
\end{aligned}
\end{equation}
where $\Delta v(x) = \sum_{y\sim x} (v(y) - v(x))$ is the Laplacian on $\Z^d$, $v_0 = 0$ is the initial {\it odometer}, and $s_0 = \eta + n  \delta_0$ is the starting sandpile. We say $s_0$ is {\it stabilizable} if there is $T < \infty$ so that 
$v_{t} = v_{T}$ for all $t \geq T$. A background is {\it robust} 
if $\eta + n  \delta_0$ is stabilizable for all $n \geq 1$, and otherwise 
is {\it explosive}. When $s_0 = \eta + n  \delta_0$ is not stabilizable, it is explosive and the infinite sequence $\{s_t\}_{t \geq 0}$ is an {\it exploding} sandpile. See Figures \ref{fig:exploding_random1} and \ref{fig:exploding_random2}.

Fey, Levine, and Peres coined these notions in \cite{fey2010growth} (see also \cite{fey2005organized}) and provided sufficient conditions for determining if a background is explosive or robust: backgrounds $\eta \leq (2d-2)$ are always robust, but otherwise can be robust or explosive, depending on the arrangement of sites with $(2d-1)$ chips. In fact, they showed that if $\eta \leq (2d - 2)$, not only 
is the background robust, but if $n$ chips are added to the origin of such $\eta$, the diameter of the set of sites which topple grows like $n^{1/d}$.

\begin{figure}[b!]
	\includegraphics[width=0.2\textwidth]{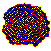}
	\includegraphics[width=0.2\textwidth]{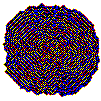}
	\includegraphics[width=0.2\textwidth]{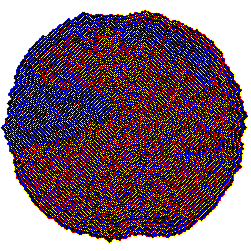}
	\includegraphics[width=0.2\textwidth]{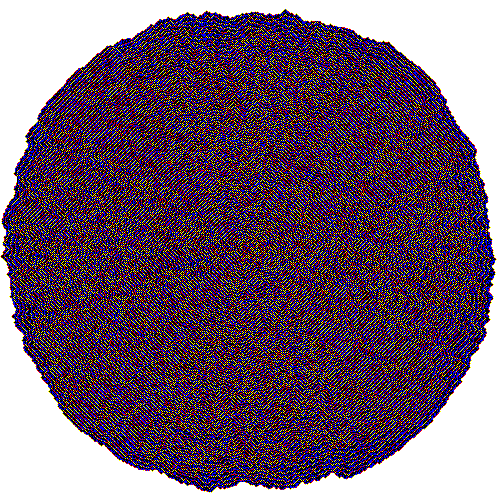}
	\caption{$s_t  1\{ v_t > 0 \}$ for $\eta \sim \mbox{Bernoulli}(3,2,1/2)$ and $t = 50, 100,250,500$.
		The color white denotes sites which haven't toppled yet, otherwise white, yellow, red, blue, and black correspond to values $0, 1,2,3, (\geq 4)$.
	}  \label{fig:exploding_random1}
\end{figure}

Pegden and Smart used this bound together with the theory of viscosity solutions to show that the terminal odometer
for $n  \delta_0$, after a rescaling, converges to the solution of a fully nonlinear elliptic PDE \cite{pegden2013convergence}.
This breakthrough then led to an explanation for the patterns which appear in two-dimensional sandpiles \cite{pegden2020stability, levine2016apollonian, levine2017apollonian}. Recently, the author used stochastic homogenization methods to extend convergence to all initial backgrounds which are stationary, ergodic, and bounded from above by $(2 d - 2)$ \cite{bou2019convergence}. 

\begin{figure}[t]
	\includegraphics[width=0.2\textwidth]{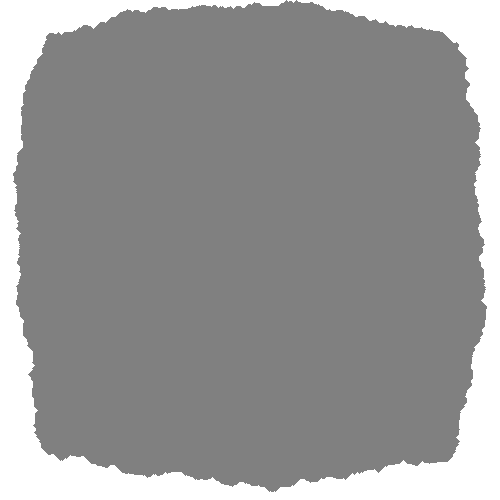}
	\includegraphics[width=0.2\textwidth]{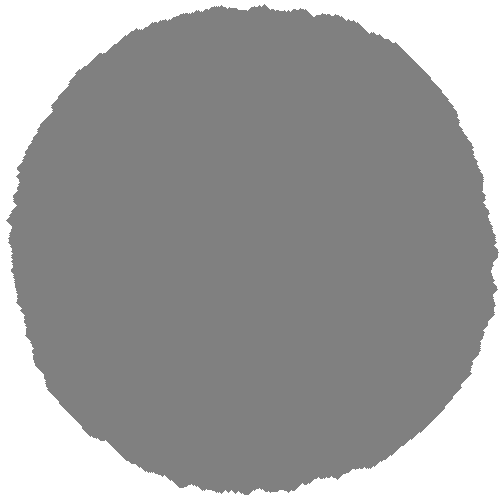}
	\includegraphics[width=0.2\textwidth]{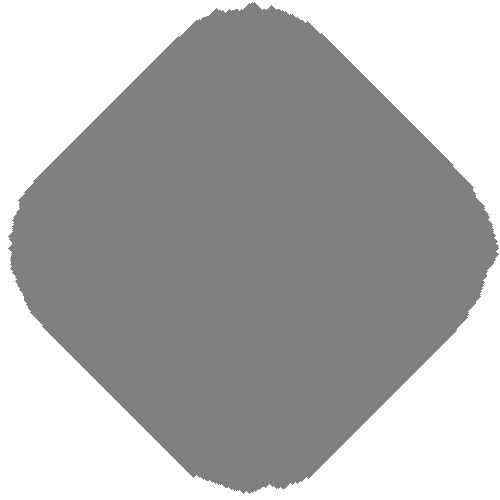} \\
	\includegraphics[width=0.2\textwidth]{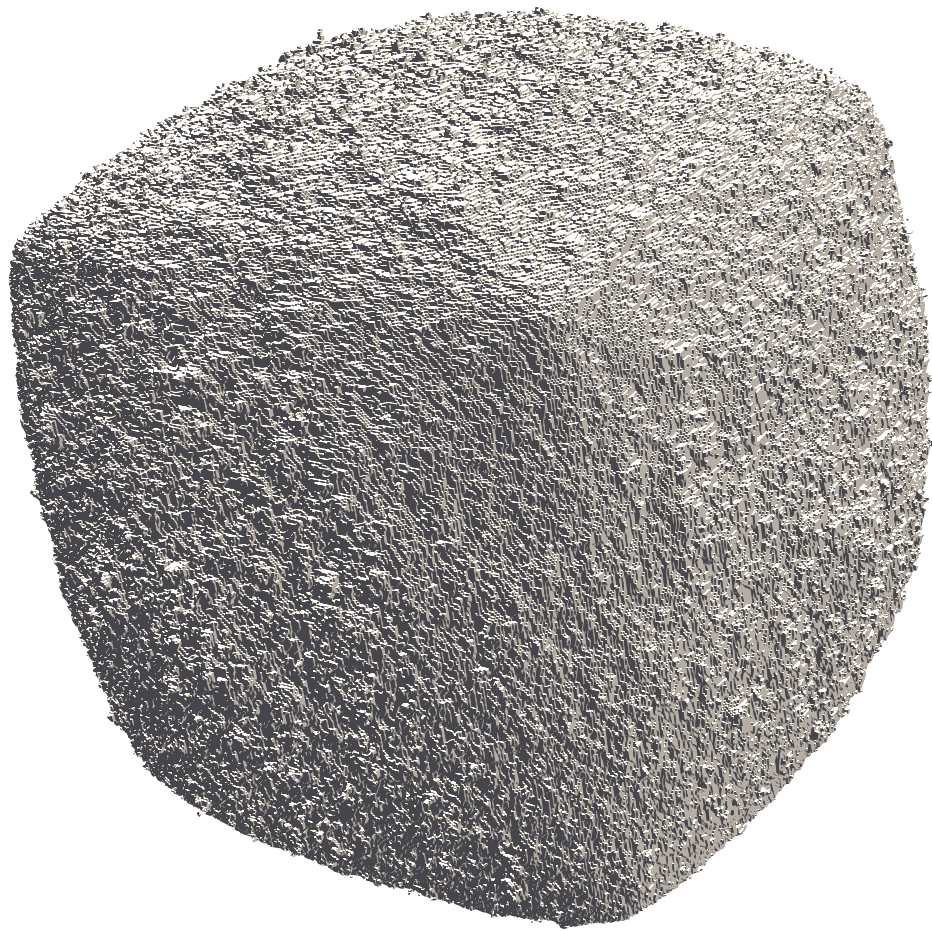}
	\includegraphics[width=0.2\textwidth]{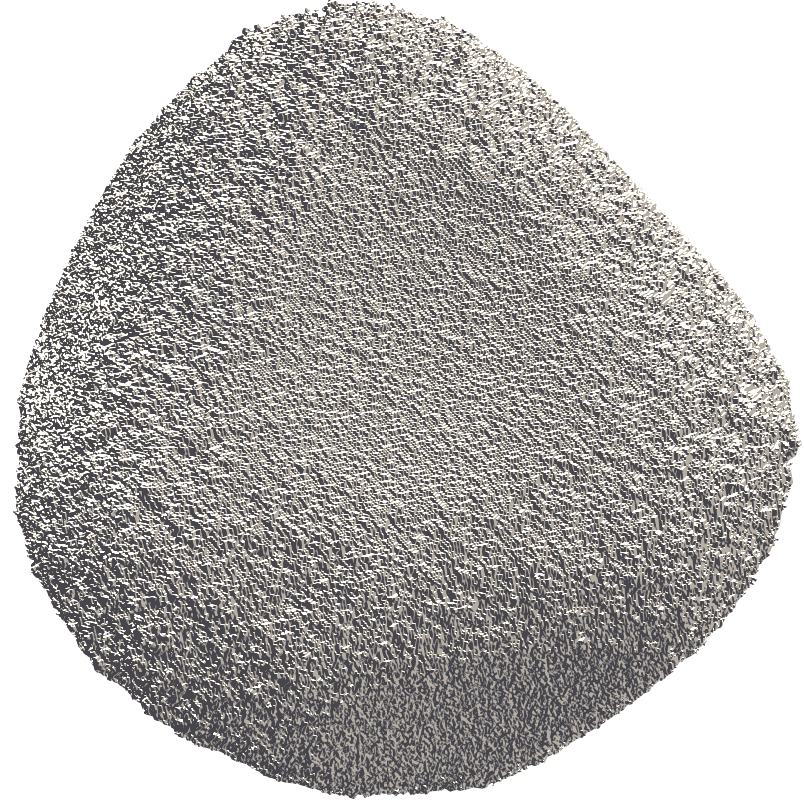}
	\includegraphics[width=0.2\textwidth]{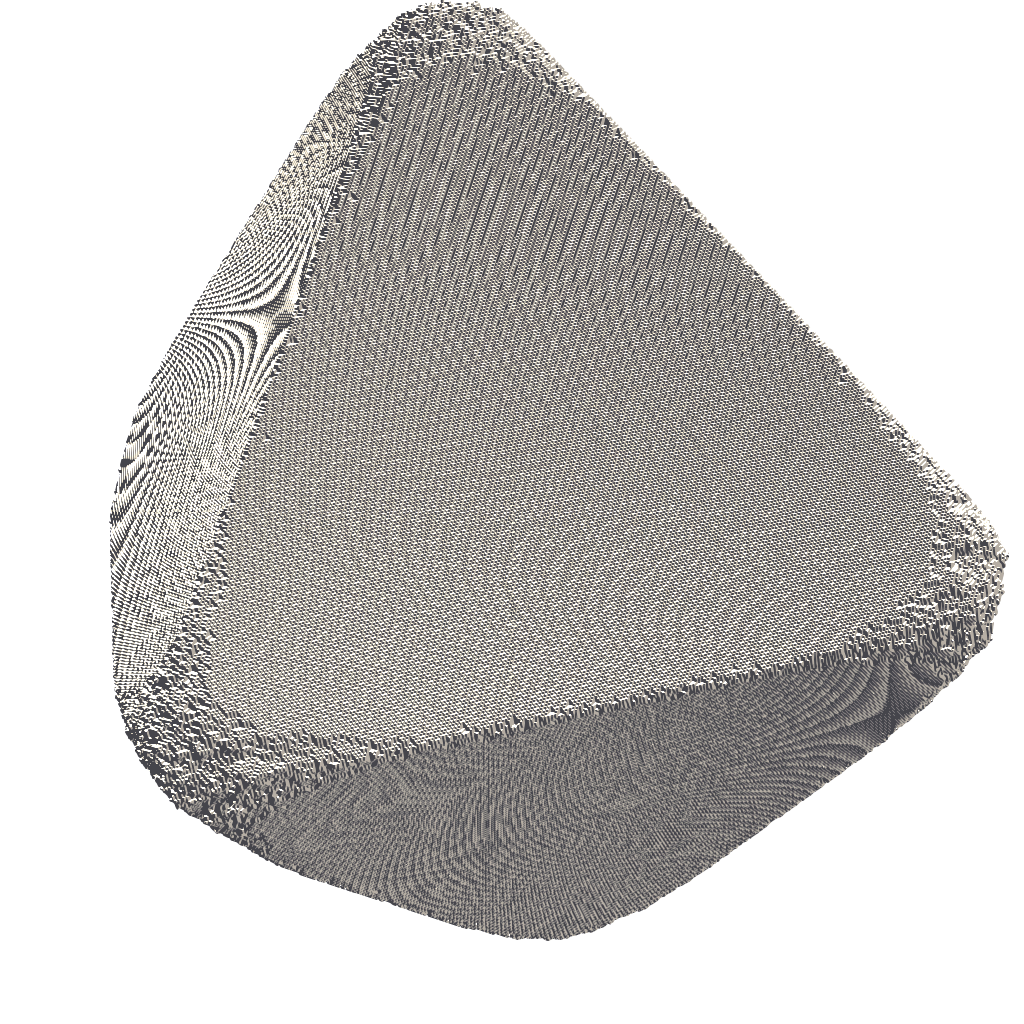}
	\caption{Snapshot of the support of an exploding sandpile $1 \{v_t > 0\}$  for $\eta \sim \mbox{Bernoulli}(2 d - 1, 2d - 2,p)$
		the first time it exits a box of side length $500$. The top row is $d = 2$ and the bottom $d = 3$. 
		From left to right, $p = 1/4, 1/2, 3/4$.}  \label{fig:exploding_random2}
\end{figure}

These results explain the phenomena of {\it scale-invariance} in sandpiles which have a {\it compact}, $n^{1/d}$, growth rate --- large, compact-growth sandpiles look like high-resolution versions of smaller sandpiles. Simple models of growth are of interest to the mathematics and physics communities ---  see, for example, \cite{dhar2013sandpile,diaconis1991growth, gravner1998cellular, packard1985two} and the references therein. The Abelian sandpile in particular has a rich history,  \cite{liu1990geometry, dhar1999abelian,le2002identity, ostojic2003patterns,  redig2005mathematical,fey2008limiting, holroyd2008chip, fey2009stabilizability, levine2009strong, levine2010sandpile, paoletti2013deterministic,levine2017laplacian, jarai2018sandpile,  klivans2018mathematics, hough2019sandpiles, aleksanyan2019discrete, lang2019harmonic, hough2019cut,chen2020laplacian, melchionna2020sandpile, alevy2020limit}.

In this paper, we study limit shapes of sandpiles in the explosive regime. The techniques used differ fundamentally from the existing compact-growth theory. Indeed, as we will demonstrate, some explosive sandpiles (both random and deterministic) do not converge. On the other hand, compact-growth sandpiles essentially always have limits --- the argument there is `soft' and applies in wide generality. Our proof below is quantitative and involves establishing specific, finite-scale estimates. We identify sufficient conditions under which exploding sandpiles converge to the level set of an asymmetric norm  --- much like in first-passage percolation  \cite{cox1981some} and threshold growth \cite{willson1978convergence, gravner1993threshold}.

\subsection{Main results}	
For expositional clarity, we consider one family of random, explosive backgrounds with limit shapes. The reader interested in generalizations may consult Section \ref{sec:generalization}. Suppose $d \geq 2$ and $\eta:\Z^d \to \Z$ is drawn from a product measure $\mathbf{P}$ with
\begin{equation} \label{eq:toy-eta-definition}
\begin{cases}
\mathbf{P}(\eta(0) = (2 d - 2)) = 1-p \\
\mathbf{P}(\eta(0) = (2 d - 1)) = p.
\end{cases}
\end{equation}
Fey, Levine, and Peres showed the following.
\begin{theorem}[Proposition 1.4 in \cite{fey2010growth}]\label{prop:fey_levine_peres_explode}
	Let $\eta$ be as in \eqref{eq:toy-eta-definition}.  If $p > 0$, $\eta$ is explosive with probability 1. 
\end{theorem}
Fix $p > 0$ and denote the (almost surely finite) {\it explosion threshold} by
\begin{equation} \label{eq:first_explosion}
M_{\eta} :=  \min\{ n \geq 1 : \mbox{ $\eta + n  \delta_0$ is not stabilizable} \}. 
\end{equation} 
We prove that the support of the infinite sequence of parallel toppling odometers, $\{v_t\}_{t \geq 1}$, for the explosive sandpile $s_0 = \eta + M_{\eta}  \delta_0$ converges under rescaling. 
\begin{theorem} \label{theorem:limit_shape}
	Let $\eta$ be as in \eqref{eq:toy-eta-definition}. 
	There exists a convex domain $\mathcal{B}(p) \subset \R^d$ so that on an event of 
	probability 1, 
	\[
	| \{ \bar{v}_t > 0\} \triangle \mathcal{B}(p)| \to 0, 
	\]
	where $\bar{v}_t(x) := v_t([t x])$. 
\end{theorem}
In fact, we prove something stronger. Not only does the support of the explosion converge, 
but the rate at which the explosion spreads also converges.

\begin{theorem} \label{theorem:convergence_of_time_constant}
	Let $\eta$ be as in \eqref{eq:toy-eta-definition} and let $T_{\eta}(x) := \min\{t \geq 1 : v_t(x) > 0 \}$. On an event of probability 1, the rescaled arrival times $\bar{T}_{\eta}(x) := n^{-1} T_{\eta}([ n x])$
	converge locally uniformly to $\mathcal{N}_{p}$, a continuous, convex, one-homogeneous function on $\R^d$. 
\end{theorem}
During our proof of Theorem \ref{theorem:convergence_of_time_constant}, we introduce a quantitative criterion for determining if a sandpile is explosive. The criterion asserts, roughly, that if a sandpile explodes quickly on a finite box, then it must do so on the entire lattice. 
We use this and a coupling with bootstrap percolation to extend the aforementioned Theorem \ref{prop:fey_levine_peres_explode}. 
\begin{theorem} \label{theorem:explosive}
	Suppose $\beta:\Z^d \to \Z$ is drawn from a product measure $\mathbf{P}$.  If $\beta \geq d$ and $\mathbf{P}(\beta(0) = (2 d - 1)) > 0$,
	then $\beta$ is explosive with probability 1.
\end{theorem}

\subsection{Proof outline}
An exploding sandpile may be thought of as a heterogeneous, discrete reaction-diffusion equation. This perspective leads us to the literature for stochastic homogenization of reaction-diffusion equations \cite{zhang2020long, feldman2019mean, lin2019stochastic, armstrong2018stochastic, caffarelli2014counter, gravner2006random, ishii1999threshold, barles1998new, gravner1996first,gravner1993threshold, willson1978convergence}. These works suggest two methods of proof. The first, which we do not pursue, is {\it half-space propagation}---a limit shape can be completely described by those starting with a half-space initial condition---an early example of this technique appears in \cite{willson1978convergence}. 
Another method is to identify a subadditive quantity similar to the first-passage time, \cite{cox1981some}, which (directly or indirectly) describes the limit shape, then apply the subadditive ergodic theorem.

Our proof of Theorem \ref{theorem:convergence_of_time_constant} follows the second outline,
however, there are several hurdles to overcome. A fundamental one is the nonlinear diffusion of the sandpile. This nonlinearity can cause explosions to propagate irregularly. In fact, an arbitrary exploding sandpile may spread quickly in certain cells, but slowly in others, causing convergence to fail. We demonstrate an explicit family of counterexamples to that effect in Section \ref{sec:counterexample}. A major part of our argument is showing that this irregularity cannot happen if $\eta \geq (2 d - 2)$ and there are `enough' sites with $(2d - 1)$ chips. 

From now until Section \ref{sec:generalization}, take $\eta$ as in the statement of Theorem \ref{theorem:convergence_of_time_constant} and fix $p > 0$. We begin in Section \ref{sec:fast_explosions} by showing that explosions on $\eta$ spread quickly. This is done by establishing a high-probability bound on the `crossing-speed' of $\eta$ in a finite, but large cube and passing to a coarsened lattice. On the coarsened lattice, there is an infinite cluster of `good cubes' upon which explosions are guaranteed to spread quickly. We use this together with
large-deviations for the chemical distance of supercritical Bernoulli percolation to get
uniform, linear bounds on the arrival times. (This portion of the proof shares some similarities with 
Schonmann's argument for bootstrap percolation \cite{schonmann1992behavior}.)

At this stage, if the arrival times, $T_{\eta}$, were subadditive, we could apply the subadditive ergodic theorem
and conclude. However, $T_{\eta}$ is not, in general, subadditive: roughly, when an explosion
started at the origin reaches some site $x$ for the first time, it mixes up the background
and so cannot be compared directly to the explosion originating from $x$. 

We overcome this lack of subadditivity by shifting our focus to a related, but simpler process, 
the {\it last-wave}---an exploding sandpile where the origin is constrained to topple a fixed number of times.  In Section \ref{sec:last_wave}, we use the established regularity of explosions to show that the last-wave can be approximated by a quantity which is exactly subadditive, and hence converges. This can also be viewed as a shape theorem for a bootstrap percolation type process. 

The proof of Theorem \ref{theorem:convergence_of_time_constant} is completed in Section \ref{sec:proof_of_theorem} where we show that the last-wave is an approximation to the expanding front of an exploding sandpile. The argument for this is a deterministic comparison which requires $\eta \geq (2d - 2)$. Then, in Section \ref{sec:generalization}, we generalize Theorem \ref{theorem:convergence_of_time_constant} 
by presenting sufficient hypotheses under which a limit shape exists. We demonstrate some need of these hypotheses
by constructing a family of (random and deterministic) exploding sandpiles which fail to have a limit shape in Section \ref{sec:counterexample}. We conclude in Section \ref{sec:criteria} with a proof of Theorem \ref{theorem:explosive}. There we indicate explicit criteria 
for determining if a sandpile is explosive. 

\subsection*{Acknowledgments}
Thank you to Charles K. Smart for motivating, helpful discussions 
during this project. Thank you to Lionel Levine for several
inspiring conversations and for suggesting this question.  Thank you to Dylan Airey for asking if exploding sandpiles have a limit shape.

\subsection*{Code}
Julia code which can compute the figures in this article is included in the arXiv upload and may be freely used
and modified.

\subsection*{Notation and conventions} 
\begin{itemize}
	\item Functions on $\Z^d$ are extended via nearest-neighbor interpolation to $\R^d$.  
	That is, if $f$ is a function on $\Z^d$, then for $x \in \R^d$, we define $f(x)$ to be $f([x])$ where $[x]$ is the nearest (breaking ties by picking the coordinate-wise minimum) point in $\Z^d$ to x. 
	\item $d$ will always refer to the dimension of the underlying space.
	\item $e_1, \ldots, e_d$ are the $d$ unit directions in $\Z^d$.
	\item $x_i = x \cdot e_i$ is the $i$th coordinate of vector $x$ and $\mathbf{x}_{d-i} = (x_{i+1}, \ldots, x_d)$.
	\item $|x|$ refers to the Manhattan norm and $|x|_{\infty}$ the $\ell_\infty$ norm.
	\item $y \sim x$ if $|y-x| = 1$. 
	\item For $x,y \in \R^d$, 
	\[
	[x,y] = [x_1, y_1]  \times \cdots \times [x_d, y_d]
	\]
	and for $a,b \in \bar{\R}$
	\[
	[a,b]^d = [a,b] \times \cdots \times [a,b]
	\]
	and for $x \in \R^d$ and $b \in \bar{\R}$, 
	\[
	[x,b]^d = [x_1,b] \times \cdots \times [x_d, b]
	\]
	and vice-versa. 
	\item The symbol $\cdot$ will sometimes be used (for visual clarity) to denote scalar multiplication. 
	\item Scalar operations on vectors/functions/sets are interpreted pointwise.
	\item  For a function  $f:A \to \R$ and a subset of its domain $\mathcal{S} \subset A$, we denote the restriction of $f$ to $\mathcal{S}$ by  $f|_{\mathcal{S}}$. 
	\item $|\cdot|$ is either the counting measure or Lebesgue measure depending on the input.
	\item For $A \subset \Z^d$, $A^c := \{ x \in \Z^d : x \not \in A\}$, $\partial A := \{x \in A^c : \exists y \in A, x \sim y\}$,
	and $\bar{A} := A \cup \partial A$. The inner boundary of $A$ is denoted by $\partial^{\circ} A  := \{ x \in A : 
	\exists y \in A^c, x \sim y\}$. 
	\item $C,c$, are positive constants which may change from line to line. Dependence is indicated by, for example, $C_d$ or $C_d(\eta)$.
	\item  The act of {\it firing} or {\it toppling} a site, $x$, removes $2 d$ chips from $x$ and adds one chip to each $y \sim x$.
	\item  $\eta \sim \mbox{Bernoulli}(a, b, p)$ is shorthand for: $\eta: \Z^d \to \Z$ is drawn from a product measure 
	$\mathbf{P}$ with $\mathbf{P}(\eta(0) =  a) = p$ and $\mathbf{P}(\eta(0) = b) = (1-p)$.
\end{itemize}

\section{Regularity of explosions} \label{sec:fast_explosions}
We use $\eta \geq (2 d- 2)$ together with the i.i.d. assumption to establish almost sure regularity of explosions. 
The main result of this section is a quantification of \cite{fey2010growth}'s Theorem \ref{prop:fey_levine_peres_explode} recalled above. 
The method is a static renormalization (see Chapter 7 in \cite{grimmett2013percolation}) inspired by Schonmann's proof for bootstrap percolation \cite{schonmann1992behavior}.

\subsection{Parallel toppling preliminaries}
Before proceeding, we mention some basic properties of parallel toppling which we use below. Recall that $\{v_t\}_{t \geq 1}$ and $\{s_t\}_{t \geq 1}$ are the infinite sequence of parallel toppling odometers and sandpiles for initial conditions $v_0 = 0$ and  $s_0 = \eta + M_{\eta}  \delta_0$.
An induction argument ((4.4) in \cite{babai2007sandpile}) shows
\begin{equation} \label{eq:parallel_toppling_recursion}
\begin{aligned}
v_{t+1}(x) &= \min \{ \lfloor \frac{ s_0(x) + \sum_{y \sim x} v_{t}(y)}{2 d} \rfloor, v_t(x) + 1\} \\
s_{t+1}(x) &= s_0(x) + \Delta v_{t+1}(x).
\end{aligned}
\end{equation}
Another induction shows that when $s_0 \leq 2 \cdot (2 d) - 1$, the minimum over the $v_t(x) + 1$ term in  \eqref{eq:parallel_toppling_recursion}
is unnecessary. 

In the sequel we also consider a version of parallel toppling $\{w_t\}_{t \geq 0}$ where 
the initial value $w_0$ is given and the odometer on some set, $\mathcal{S}$, (the complement of a cube, the origin) is `frozen' at $w_0$ 
and the initial sandpile is such that the topplings $w_0$ were performed: $s_0' = s_0 + \Delta w_0$ and, 
\begin{equation} \label{eq:frozen_parallel_toppling}
\begin{aligned}
w_{t+1}(x) &=
\begin{cases}
w_{t}(x) + 1 \{ s_{t}'(x) \geq 2 d \}  &\mbox{if $x \not \in \mathcal{S}$} \\
w_0(x)  &\mbox{if $x \in \mathcal{S}$}
\end{cases} \\
s_{t+1}' &= s_{t}' + \Delta (w_{t+1} - w_{t}).
\end{aligned} 
\end{equation}
We call this {\it $\mathcal{S}$-frozen} parallel toppling. (Recently, and after this paper was written, \cite{goles2021freezing} was posted---therein so-called `freezing sandpiles' are studied in the context of computational complexity.) 
If $s_0 + \Delta w_0 \leq 2 \cdot (2 d)-1$ on $\mathcal{S}^c$, then as above:
\begin{equation} \label{eq:frozen_parallel_toppling_recurse}
\begin{aligned}
w_{t+1}(x) &=   \lfloor \frac{ s_0(x) + \sum_{y \sim x} w_{t}(y)}{2 d} \rfloor  \qquad \mbox{for $x \not \in \mathcal{S}$} \\
s_{t+1}'(x) &= s_0(x) + \Delta (w_{t+1} - w_{0})(x).
\end{aligned}
\end{equation}
Also, if $s_0 \leq (2 d-1)$ and $w_0 \equiv 0$ on $\mathcal{S}^c$, then $w_t |_{\mathcal{S}^c} \leq \max_{x \in \partial^{\circ} \mathcal{S}} w_0(x)$ for all $t \geq 0$.
The definitions allow us to compare the two versions of parallel toppling,

\begin{equation} \label{eq:odometer_wave_comp}
\begin{aligned}
v_{t + t_0} \geq w_t  \qquad &\mbox{where $t_0:=  \min\{ t \geq 0 : v_{t} \geq w_0 \}$} \\
w_t \geq v_t  \qquad &\mbox{ if $w_0 \geq v_0 \equiv 0$ and $w_0|_{\mathcal{S}} \geq \sup_{t} (v_t|_{\mathcal{S}})$}.
\end{aligned}
\end{equation}

\subsection{Crossing speeds}
To provide a global upper bound on the arrival times, $\mathcal{T}_{\eta}(x)$, we show a local upper bound. In particular 
we study the following `cell problem', a term from homogenization denoting 
a simple problem which describes the local behavior of a more complicated one. 

We consider sandpile dynamics on a box of side length $k$, $Q_k := \{x \in \Z^d : 1 \leq x \leq k\}$. For a point $z \in \bar{Q}_k$ and direction $1 \leq i \leq d$, denote the line passing from one side of the
box to the other 
\begin{equation}
\mathcal{L}_k^{(i,z)} := \bigcup_{j=1,\ldots,k} (z_1, \ldots, z_{i-1}, j, z_{i+1}, \ldots, z_d).
\end{equation}
Let $w_t$ be the parallel toppling odometer for $\{Q_k^c \cup \mathcal{L}_k^{(i,z)}\}$-frozen parallel toppling (defined in \eqref{eq:frozen_parallel_toppling})
with initial conditions $w_0(x) = 1\{x \in \mathcal{L}_k^{(i,z)}\}$, $s_0' = \eta$. Denote the {\it crossing time}
\begin{equation}
\mathfrak{C}_{k}^{(i,z)} := 
\begin{cases}
\min \{ t \geq 1 : w_t|_{Q_k} = 1 \} \qquad &\mbox{if $w_{\infty} |_{Q_k} = 1$}  \\
\infty \qquad &\mbox{otherwise}.
\end{cases}
\end{equation}
We show that if $k$ is sufficiently large, the crossing time is bounded with high probability.

\begin{prop} \label{prop:crossing_time}
	For every $\delta > 0$, there is a $k$ so that  
	\begin{equation} \label{eq:crossing_time_bound}
	\max_{i,z} \mathfrak{C}_{k}^{(i,z)} \leq k^d
	\end{equation}
	with probability at least $(1-\delta)$. 
\end{prop}
\begin{proof}
	For each $k \geq 1$, we construct an event for which \eqref{eq:crossing_time_bound} occurs 
	with probability approaching 1 in $k$.  By Harris' inequality (see, \eg, \cite{alon2016probabilistic}) and symmetry,
	it suffices to show \eqref{eq:crossing_time_bound} for lines in one direction, say $i=1$.
	Let $z \in \bar{Q}_k$ be given. 
	
	Write $\mathfrak{C}_k =  \mathfrak{C}_k^{(1,z)}$.  We show that if all lines, $\mathcal{L}_k^{(1,y)}$, $y \in Q_k$, contain at least one 
	site with $(2 d - 1)$ chips, then every site in the cube eventually topples.  Denote the event upon which this happens by, 
	\begin{equation}
	\Omega' := \bigcap_{y \in Q_k} \Omega_y := \bigcap_{y \in Q_k} \{\eta:\Z^d \to \Z : \eta(x) = (2 d-1) \mbox{ for some $x \in \mathcal{L}_k^{(1,y)}$} \}.
	\end{equation}
	Recall $p > 0$ is the probability of a site having $(2 d-1)$ chips. Fix $0 < \epsilon < p$, and note, by Hoeffding's inequality (see, \eg, \cite{alon2016probabilistic}), 
	for each $y \in Q_k$, 
	\[
	\mathbf{P}( \sum_{ x \in \mathcal{L}_k^{(1,y)} }  1(\eta(x) = 2 d -1) \leq (p-\epsilon) k) \leq  \exp(-2 \epsilon^2 k).
	\]
	Therefore, by the union bound, (deleting duplicates),
	\[
	\mathbf{P}(\Omega'^c) \leq k^{d-1} \mathbf{P}(\Omega_1'^c) \leq k^{d-1} \exp(-2 \epsilon^2 k),
	\]
	and so for every $\delta > 0$, there is $k$ sufficiently large so that 
	$P(\Omega') \geq 1 - \delta$.
	
	It remains to check that for $\eta \in \Omega'$, $\mathfrak{C}_k \leq k^d$. We do so by constructing a toppling procedure which is dominated by $w_t$. After firing $\mathcal{L}_k^{(1,z)}$, all sites in neighboring lines, $\mathcal{L}_k^{(1,y)}$, $\mathbf{y}_{d-1} \sim \mathbf{z}_{d-1}$ have at least $(2 d - 1)$ chips. In fact, since $\eta \in \Omega'$, at least one site in each neighbor 
	$\mathcal{L}_k^{(1,y)} \subset Q_k$ has $2 d $ chips, causing all sites in the line to topple. Iterating shows this procedure will terminate with every site in $Q_k$ toppling in at most $k^{d}$ steps. 
	
\end{proof}

\subsection{A static renormalization scheme}
We exhibit a coarsening of the lattice upon which explosions are guaranteed to spread quickly. A cube, $Q_k$, is {\it good} if 
$\max_{i,z} \mathfrak{C}_{k}^{(i,z)} \leq k^d$. For each $i \in \Z^d$, let 
\begin{equation} 
Q_k(i) := Q_k + i \cdot k.
\end{equation}
The cubes $\{Q_k(i)\}_{i \in \Z^d}$ define a {\it macroscopic lattice} with edge set $\{ (Q_k(i), Q_k(j)), |j-i| = 1\}$. 
For $k$ sufficiently large, Proposition \ref{prop:crossing_time} implies that the set of good cubes dominates a high density independent site percolation process on the macroscopic lattice.  This together with large deviations bounds for supercritical percolation \cite{antal1996chemical, garet2007large} imply the following. (See, for example, Section 5 in \cite{mathieu2008quenched} for an explicit proof.)

\begin{prop} \label{prop:renormalization_properties}
	For fixed $k$ large enough, there are constants $c, C$ so that the following hold
	on an event of probability 1. 
	\begin{enumerate}
		\item There is a unique infinite cluster $\mathcal{C}_{\infty}$ of good cubes
		on the macroscopic lattice $\{Q_k(i)\}_{i \in \Z^d}$. 
		\item	There is $n_0$ so that for $n \geq n_0$, any connected
		component of $\mathcal{C}_{\infty}^c$ that intersects $[-n,n]^d$ has volume 
		smaller than $(\log n)^{5/2}$.
		\item   There is $n_0$ so that for $n \geq n_0$, for any $x,y \in \mathcal{C}_{\infty}$ with $|x| \leq n$ and $|x-y| \geq (\log n)^2$, 
		\[
		c |x-y| \leq d(x,y) \leq C |x-y|,  
		\] 
		where $d$ is the chemical (graph) distance on $\mathcal{C}_{\infty}$.
	\end{enumerate}
\end{prop}

The definition of $\mathcal{C}_{\infty}$ ensures that once $Q_k(i) \in \mathcal{C}_{\infty}$ is {\it overlapped}
by the support of the odometer --- $Q_k(i) \cap 1 \{ v_t > 0 \}$ contains a straight line --- an explosion will occur. This together 
with Proposition \ref{prop:renormalization_properties} controls the speed at which the explosion propagates.
We show next that the explosion spreading in $\mathcal{C}_{\infty}$ also quickly fills holes in the cluster.

\subsection{A path-filling property}
For a set of points $A \subset \Z^d$, 
let $m_i := \min_{z \in A} z_i$ and $M_i := \max_{z \in A} z_i$ for $i = 1, \ldots, d$. 
Denote the bounding rectangle of $A$ as 
\begin{equation}
\br(A) := \{ z \in \Z^d  : \mbox{ $m \leq z \leq M$} \}.
\end{equation}
We show, using $\eta \geq (2 d - 2)$, that if the odometer is strictly positive
on a path of points at some time then eventually the odometer is strictly positive on the
bounding rectangle of that path. Essentially, if not then the support of the terminal odometer must have a corner, \ie, an untoppled site with two neighbors which have toppled, a contradiction. Our proof uses this idea together 
with a slightly technical induction (which, it seems, we cannot avoid as the claim is needed in all dimensions).
See Figure \ref{fig:fill_in} for an illustration of this result. 

\begin{figure}[bt]
	\includegraphics[width=0.25\textwidth]{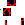}
	\caption{Terminal $A$-frozen parallel toppling odometer, $w_{\infty}$, on $\Z^2$ for initial conditions $w_0 = 1\{x \in A\}$ and $s_0' = 2$. The set $A$ is denoted by red pixels and black pixels are sites which eventually topple.}
	\label{fig:fill_in}
\end{figure}

\begin{lemma}\label{lemma:local_convexity}
	Let $A := \{z^{(i)}\}$ be a finite path $z^{(i)} \sim z^{(i+1)}$.
	Let $w_t$ denote $A$-frozen parallel toppling with initial conditions 
	$w_0|_{A} = 1$, $s_0' = \eta$. Then, $w_t|_{\br(A)} = 1$, for all $t \geq |\br(A)|$.   
\end{lemma}
\begin{proof}
	By monotonicity of parallel toppling, we may take $\eta = (2 d - 2)$. Moreover, 
	it suffices to show that every site in $\br(A)$ eventually topples, as if no site topples at time $t$, then no site topples at time $(t+1)$.
	
	We say $A$ contains a $( \pm i, \pm j)$ turn if $z^{(m)} = z^{(m-1)} \pm  e_i$ and $z^{(m+1)} = z^{(m)} \pm e_j$
	for some $i \not = j$ where $z^{(m-1)} \sim z^{(m)} \sim z^{(m+1)}$ are in $A$.
	If $A$ does not contain a turn, then $\br(A) = A$. Hence, we may suppose it contains at least one turn.
	
	\subsubsection*{Case 1 - one-turn path}
	By shifting coordinates, we may suppose $A$ contains only a $(1,2)$ turn and that 
	\begin{equation}
	A = \{ 0, e_1, \ldots, k_1 e_1, k_1 e_1 + e_2, \ldots, k_1 e_1 + k_2 e_2 \}
	\end{equation}
	for $k_1 \geq k_2$.	We induct on $k_2$. If $k_2 = 1$, then after firing every site in $A$, 
	all sites in $A + e_2$ get one chip, while the corner site,
	$((k_1 - 1) e_1 + e_2)$ gets 2 chips. Since $\eta = (2 d- 2)$, that corner becomes unstable and fires, 
	causing all of its neighbors to the left, $((k_1 - 2) e_1 + e_2) \sim ((k_1 - 3) e_1 + e_2) \sim \cdots \sim e_2$ to fire. Continuing the induction shows that every site in $\br(A) = \{ x \in \Z^d : 0 \leq x \leq (k_1 e_1 + k_2 e_2)\}$ eventually fires. 
	
	\subsubsection*{Case 2 - cubic path} 
	We call $A$ a {\it cubic} path if, after an isometry, 
	\begin{equation}
	A = \{ 0, e_1, \ldots, k_1 e_1, k_1 e_1 + e_2, \ldots, \sum_{i=1}^d k_i e_i \}
	\end{equation}
	for $k_1 \geq \cdots \geq k_d \geq 0$. Let $d_0 := \max \{ i \leq d : k_i > 0 \}$. We induct on $d_0$, the base case $d_0 = 2$ established in Case 1. 
	For notational convenience, suppose the claim holds for $d_0 = (d-1)$ and we verify it for $d_0 = d$.

	Consider the $(d-1)$-turn subpaths, 
	\begin{equation}
	\begin{aligned}
	\mathcal{P}_1 &:= \{0, e_1, \ldots, \sum_{i=1}^{d-1} k_i e_i \} \\
	\mathcal{P}_2 &:= \{k_1 e_1, k_1 e_1 + e_2 , \ldots, \sum_{i=1}^{d} k_i e_i \}.
	\end{aligned}
	\end{equation}
	By the inductive hypothesis, after $\mathcal{P}_i$ fire,  
	both $(d-1)$-dimensional faces, 
	\begin{equation}
	\begin{aligned}
	\mathcal{F}_1 &:= \{x \in \Z^d : 0 \leq x \leq \sum_{i=1}^{d-1} k_i e_i \} \\
	\mathcal{F}_2 &:= \{x \in \Z^d : k_1 e_1 \leq x \leq \sum_{i=1}^{d} k_i e_i \}
	\end{aligned}
	\end{equation}
	fire. We then `fill in' the cube by identifying newly fired $(d-1)$-turn paths:
	\begin{equation}
	\begin{aligned}
	\mathcal{P}'_j := \{ & j  e_2, j e_2 + e_1, \ldots, j e_2 + k_1 e_1,  \\ 
	& j e_2 + k_1 e_1 + e_3, \ldots, j e_2 + k_1 e_1 + k_3 e_3, \\
	& \cdots \\
	&j e_2 + \sum_{i\not=2} k_i e_i \} 
	\end{aligned}
	\end{equation}
	which are in $\mathcal{F}_1 \cup \mathcal{F}_2$ for $j=0, \ldots, k_2$. By the inductive hypothesis, the firing of each $\mathcal{P}'_j$ makes every $(d-1)$-dimensional layer,
	\[
	\mathcal{L}_j := \{x \in \Z^d : j e_2 \leq x \leq j e_2 + \sum_{i\not=2} k_i e_i \},
	\]
	fire and $\br(A) = \cup_{j=0}^{k_2} \mathcal{L}_j$.

	\subsubsection*{Case 3 - general path}
	
	It suffices to show that if there is a path of firings between any two distinct points $x,y$, 
	then $\br(\{x,y\})$ eventually fires. Before showing this, we suppose it were true and demonstrate sufficiency.
	Take $q \in \br(A)$ and observe by definition there are points $(z^{(1)}, Z^{(1)}), \ldots, (z^{(d)}, Z^{(d)})$ in $A$
	with $z_i^{(i)} \leq q_i \leq Z_i^{(i)}$.  Then,
	\[
	q^{(1)} := (q_1, \mathbf{q}_{d-1}') \in \br(\{z^{(1)}, Z^{(1)}\})
	\]
	for some $(d-1)$-vector $\mathbf{q}_{d-1}'$. Continue and let 
	\[
	q^{(2)} := \begin{cases}
	(q_1, q_2, \mathbf{q}_{d-2}'')  \in \br(\{q^{(1)}, z^{(2)}\}) &\mbox{ if $q^{(1)}_2 \geq q_2$}\\
	(q_1,q_2, \mathbf{q}_{d-2}''')  \in \br(\{q^{(1)}, Z^{(2)}\}) &\mbox{ otherwise }\\
	\end{cases}	
	\]
	for some $(d-2)$-vectors $\mathbf{q}_{d-2}'', \mathbf{q}_{d-2}'''$. After iterating, we find $q^{(d)} = q$, which shows that eventually $q$ will fire.

	Now fix two points $x,y \in A$ and decompose
	a path between them into a sequence of cubic paths
	\[
	\mathcal{P}^{(1)} := \bigcup_{i=1}^m \mathcal{P}^{(1)}_i, 
	\]
	where $\mathcal{P}^{(1)}_i := \{p_{i-1}, \ldots, p_i \}$ is cubic and $p_0 := x$ and $p_{m} := y$. (This can be done by, for example, 
	starting at x and exploring the path but cutting whenever the cubic condition is violated.) Case 2 shows that eventually every site in $\mathcal{A}^{(1)} := \bigcup_{i=1}^m \br(\mathcal{P}^{(1)}_i)$ will fire. 
	If $m = 1$, we are done, otherwise we construct a new cubic path from $p_0$ to $p_2$ 
	passing through $\mathcal{A}^{(1)}$. Once we have shown this, we
	iterate to conclude.

	Suppose $p_0 = 0$, $p_{1} = \sum_{j=1}^d k_j e_j$, and 
	\[
	p_{2} = \sum_{j=1}^{d_1} (k_j-k_j') e_j  + \sum_{j=(d_1+1)}^{d_2} (k_j - k_j') e_j + \sum_{j=(d_2+1)}^{d} (k_j+k_j') e_j,
	\]
	for some $1 \leq d_1 \leq d_2 \leq d$ and $k_j,k_j' \geq 1$ where $(k_j - k_j') < 0$ for $j \leq d_1$ and $(k_j - k_j') \geq 0$
	for $(d_1 + 1) \leq j \leq d_2$. After this coordinate change, it suffices to exhibit a path from $p_0$ to $p_2$ with differences constrained to be $-e_j$ for $j = 1, \ldots, d_1$ and 
	$+e_j$ for $j = (d_1 + 1), \ldots, d$.
	
	There is a cubic path (only positive moves) from $p_0$ to
	\[
	w_1 := \sum_{j=(d_1+1)}^{d_2} (k_j - k_j') e_j + \sum_{j=(d_2+1)}^{d} (k_j) e_j
	\]
	contained within $\br(\{p_0, p_1\})$ as $p_0 = 0 \leq w_1 \leq p_1$. Then, since $w_1 \in \br(\{p_1,p_2\})$ there is a cubic path (only positive moves) from $w_1$ to 
	\[
	w_2 := w_1 + \sum_{j=(d_2+1)}^{d} k_j'e_j
	\]
	contained in $\br(\{p_1, p_2\})$ and similarly there is a cubic path (only negative moves) from $w_2$ to 
	\[
	p_2 = w_2  + \sum_{j=1}^{d_1} (k_j-k_j') e_j.
	\]
	Our new cubic path is the concatenation of these three paths: $p_0 \to w_1  \to w_2 \to p_2$. 
\end{proof}

\section{The last-wave}\label{sec:last_wave}
In this section we study a simplified parallel toppling procedure closely related to bootstrap percolation (see Section \ref{sec:criteria}
for an explicit connection, we do not utilize the coupling here). This simplified process has an inherent subadditive structure
which allows us to prove convergence using the subadditive ergodic theorem. In the next section we show that this process is a good approximation to an exploding sandpile.

\subsection{\texorpdfstring{The $n$-wave process}{The n-wave process}} \label{subsec:n_wave}
Fix $n \geq 1$, $z \in \Z^d$, and consider the {\it $n$-wave} for $\eta$ starting at $z$,
\begin{equation} \label{eq:n_wave}
\begin{aligned}
u_0^{(n,z)} &:= n  \delta_z \\
u_{t+1}^{(n,z)}(x) &:= \lfloor \frac{ \sum_{y \sim x}  u_{t}^{(n,z)}(y) + \eta(x)}{2 d} \rfloor \mbox{ for $x \not= z$}.
\end{aligned}
\end{equation}
Note that there is intentionally no minimum in \eqref{eq:n_wave} with $u^{(n,z)}_t(x) + 1$
and it is possible that $u_{t+1}^{(n,z)}(x) > u_t^{(n,z)}(x) + 1$.
In particular, the $n$-wave started at $z$ dominates the $\{z\}$-frozen parallel toppling process defined in \eqref{eq:frozen_parallel_toppling} with the same initial conditions but may not coincide with it.

Overloading terminology, 
the $n$-wave is {\it stabilizable} if there is $T < \infty$
so that $u^{(n,z)}_{t} = u^{(n,z)}_{T}$ for all $t \geq T$. Let 
\begin{equation}
\hat{M}_{\eta}(z) := \min \{ n \geq 1 : \mbox{the $n$-wave for $\eta$ starting at $z$ is not stabilizable}\}.
\end{equation}

We write $u_t^{(z)} := u_t^{(\hat{M}_{\eta}(z),z)}$ for the $\hat{M}_{\eta}(z)$-wave starting at $z$ and call this the {\it last-wave}.  We also consider
the {\it penultimate-wave} starting at $z$ as the terminal odometer for the $(\hat{M}_{\eta}(z) - 1)$-wave, 
$\tilde{u}^{(z)}$, defined to be the zero function when $\hat{M}_{\eta}(z) = 1$. The set of sites touched by the penultimate wave is its {\it penultimate-cluster},
\begin{equation} \label{eq:penultimate_cluster}
\mathcal{P}(z) := \{z\} \cup \{ x \in \Z^d : \mbox{ there is $y \sim x$ with $\tilde{u}^{(z)}(y) > 0 $}\}
\end{equation}
(we included the point, $\{z\}$, as $\hat{M}_{\eta}(z)$ may be 1). When $z$ is the origin, we omit the superscripts.

The {\it arrival time} for the last-wave starting at site $y$ to site $x$ is
\begin{equation} \label{eq:lastwave_fpp}
\hat{T}(y,x) := \min\{ t \geq 1 : u_{t}^{(y)}(x) > 0\}
\end{equation}
and the {\it penultimate-cluster arrival time} is 
\begin{equation} \label{eq:lastwave_modified_fpp}
\tilde{T}(y,x) := \min\{ t \geq 1 : u_{t}^{(y)}|_{\mathcal{P}(x)} > 0\}. 
\end{equation}
Note that the inequality in \eqref{eq:lastwave_modified_fpp} is strict --- $\tilde{T}(y,x)$ is the first time
the support of the last wave started at $y$ contains the entire penultimate wave of $x$, $\mathcal{P}(x)$. 
We write $\hat{T}(x) := \hat{T}(0,x)$ and $\tilde{T}(x) := \tilde{T}(0,x)$.
(The choice of the same letter $T$ for all arrival times is intentional --- we will see they are asymptotically close.)

\subsection{Basic properties of the last-wave} \label{sec:last_wave_basic_properties}
We derive some basic properties of the last-wave. Throughout this section and the next, let $\eta$ be drawn from the event of probability
1 in Proposition \ref{prop:renormalization_properties} and let $k$ be the (deterministic but large) side length of a good cube.  The following is a consequence of Theorem 4.1 in \cite{fey2008limiting}.
\begin{lemma} \label{lemma:square_lower_bound}
	There is a constant $C := C_d$ so that for all $n \geq 1$, the support of the $(C n^d)$-wave contains $[-n,n]^d$. 
\end{lemma}
\begin{proof}
	From Theorem 4.1 in \cite{fey2008limiting}, we know that if $n$ chips are placed at the origin 
	on a background of $(2d - 2)$, then the support of the terminal odometer, $\tilde{v}$, contains a cube of radius 
	$r_n \geq (n^{1/d}-3)/2$. Moreover, $\tilde{v}(0) \leq C n^d$. This implies, by \eqref{eq:odometer_wave_comp} that a $(C n^d)$-wave contains $[-n,n]^d$.
\end{proof}

\begin{lemma} \label{lemma:finite_exploding}
	The last-wave is well-defined, $\hat{M}_{\eta}(0) < \infty$.
\end{lemma}
\begin{proof}
	By Lemma \ref{lemma:square_lower_bound}
	and Proposition \ref{prop:renormalization_properties}, if the origin is fired a sufficient number of times, the support of the odometer contains a good cube, $Q_k \subset \mathcal{C}_{\infty}$.
\end{proof}

\begin{lemma} \label{lemma:last_wave_bound}
	The last wave started at $z$ is bounded by one outside the interior of the penultimate-cluster of $z$. Moreover, for all $t \geq 1$ and $x,z \in \Z^d$, 
	\[
	u_t^{(z)}(x) \leq (1 + \tilde{u}^{(z)}(x)).
	\]
\end{lemma}

\begin{proof}
	To reduce clutter, we take $z = 0$. We prove this by induction on $t$. 
	The base case $t=0$ follows by definition. For all $t \geq 1$, the definition also ensures it holds at the origin.
	So, we may take $x \not = 0$ and check:
	\begin{align*}
	u_{t+1}(x) &= \lfloor \frac{ \sum_{y \sim x}  u_{t}(y) + \eta(x)}{2 d} \rfloor \\
	&\leq \lfloor \frac{ \sum_{y \sim x} (1 + \tilde{u}(y)) + \eta(x)}{2 d} \rfloor \\
	&= 1 + \tilde{u}(x) + \lfloor \frac{ \sum_{y \sim x} (\tilde{u}(y)-\tilde{u}(x)) + \eta(x)}{2 d} \rfloor \\
	&= 1 + \tilde{u}(x)
	\end{align*}
	as $\Delta \tilde{u}(x) + \eta(x) \leq (2 d- 1)$ for $x \not = 0$.

\end{proof}

\begin{lemma} \label{lemma:approx_subadditivity}
	The penultimate-cluster arrival times are subadditive: for all $a,b,z \in \Z^d$, 
	\[
	\tilde{T}(a,z) \leq \tilde{T}(a,b) + \tilde{T}(b, z).
	\]	
\end{lemma}
\begin{proof}
	Suppose $\mathcal{P}(z) \not \subseteq \mathcal{P}(b)$, otherwise the claim is immediate.
	It suffices to check  
	\begin{equation}
	w_t(x)  := u_{\tilde{T}(a,b)+t}^{(a)}(x) \geq u_{t}^{(b)}(x) \qquad \mbox{for all $t \geq 1$ and $x \in \mathcal{P}(b)^c$},
	\end{equation}
	which we do by induction. 
	By Lemma \ref{lemma:last_wave_bound}, $u_t^{(b)}(x) \leq 1$ if $\tilde{u}^{(b)}(x) = 0$. 
	Also, by definition of the penultimate-cluster, we have that $\tilde u^{(b)}$  is zero 
	on $\partial^{\circ} \mathcal{P}(b)$.
	Hence, for all $t \geq 1$ and $x \in \partial^{\circ} \mathcal{P}(b)$,  $u_t^{(b)}(x) \leq 1 \leq w_t(x)$.
	Using this and the inductive hypothesis, if $x \in \mathcal{P}(b)^c \cap \{a\}^c$, 
	\begin{align*}
	w_{t+1}(x) &= \lfloor \frac{ \sum_{y \sim x} w_t(y) + \eta(x) }{ 2d} \rfloor \\
	&\geq  \lfloor \frac{ \sum_{y \sim x} u_{t}^{(b)}(y) + \eta(x) }{ 2d} \rfloor \\
	&=  u_{t+1}^{(b)}(x).
	\end{align*}
	If $x \in \mathcal{P}(b)^c \cap \{a\}$, then $w_t(x) \geq 1 \geq u_t^{(b)}(x)$ as $\hat{M}_{\eta}(a) \geq 1$.

\end{proof}

\begin{lemma} \label{lemma:last_penultimate_cluster}
	There is a constant $\gamma > 0$ so that for all $n$ sufficiently large 
	and $|x| \leq n$, 
	\[
	\mathcal{P}(x) \subset x + [-(\log n)^{\gamma}, (\log n)^{\gamma}]^d.
	\]	
\end{lemma}
\begin{proof}
	By Lemma \ref{lemma:local_convexity}, $1\{ \tilde{u}^{(x)} > 0 \}$ is a rectangle,
	therefore it suffices to bound the maximal side length. By Proposition \ref{prop:renormalization_properties},
	for all $n$ sufficiently large, if any side length of the rectangle exceeds  $(2 k \log n)^{5/2}$ then it
	must overlap a good cube, contradicting stability. 
\end{proof}

\begin{lemma} \label{lemma:penultimate_upper_bound}
	There are constants $\gamma$ and $C$ so that on an event of probability 1, for all $n$ sufficiently large and $|x| \leq n$, 
	\[
	\hat{T}(x) \leq  C |x| +  (\log n)^{\gamma} .
	\]
\end{lemma}
\begin{proof}
	Let $x \in \Z^d$ be given. Since the sandpile is exploding, at some constant time $C_\eta$
	the support of the odometer will overlap the infinite cluster at a good cube near the origin, $Q_k(z)$ for some $z \in \Z^d$. Once this occurs the arrival time to any site is at most a constant times the chemical distance in the infinite cluster. Let $Q_k(y)$ for $y \in \Z^d$ be one of the nearest cubes in $\mathcal{C}_{\infty}$ to $x$.  There are now two cases to consider.

	If $|z-y| < (\log n)^2$, then we may choose nearby points $M_{i} \in \mathcal{C}_{\infty}$ so that $(\log n)^{c} \leq d(z,M_i) \leq (\log n)^C$ and $x,y \in \br(\{M_i\})$. 
	If $|z-y| \geq (\log n)^2$, then by the chemical distance bound and the definition of $\mathcal{C}_{\infty}$, 
	within at most $C |z-y| \leq C(|z| + |y|)$ steps, $Q_k(y)$ will topple. 
	Once this happens, $\mathcal{P}(x)$ is surrounded in at most $(\log n)^C$ more steps and $\mathcal{P}(x)$ will fire in at most $|\mathcal{P}(x)|$ additional steps.
	
\end{proof}

\subsection{Convergence of the last-wave}
We show that the arrival time for the last-wave converges under rescaling. 

\begin{lemma} \label{lemma:last_penultimate_close}
	There exists a constant $\gamma > 0$ so that on an event of probability 1, for all $n$ sufficiently large and $|x| \leq n$,
	\[
	\hat{T}(x) \leq \tilde{T}(x) \leq \hat{T}(x) + (\log n)^{\gamma}.
	\]
\end{lemma}
\begin{proof}
	The first inequality is immediate.  For the second inequality, let $C$ and $n_0$ be as in Proposition \ref{prop:renormalization_properties}.
	Take $n \geq n_0$ and suppose $v_t(x) > 0$ for $\min(t, |x|) > C$ and $|x| \leq n$.  We must show that there is a nearby good cube $Q_k(y) \subset \mathcal{C}_{\infty}$
	which has already fired. Once we have shown this, the same argument as in Lemma \ref{lemma:penultimate_upper_bound}
	allows us to conclude. 	This is, however, a consequence of Lemma \ref{lemma:local_convexity} and Proposition \ref{prop:renormalization_properties}. Any path of topplings of length at least $( 2k \log n)^{5/2}$ must overlap a good cube. 
\end{proof}

\begin{prop} \label{prop:convergence_of_the_last_wave}
	On an event of probability 1, the rescaled last-wave arrival times 
	\[
	n^{-1} \hat{T}_{\eta}([n x])
	\] 
	converge locally uniformly to $\mathcal{N}_{p}$, a continuous, convex, one-homogeneous function on $\R^d$. 
\end{prop}
\begin{proof}
	In light of Lemma \ref{lemma:last_penultimate_close} it suffices to prove the result
	for $\tilde{T}$. Convergence in integer directions follows from the subadditive ergodic theorem. 
	Everywhere convergence then follows from continuity and approximation. The properties of $\mathcal{N}_p$
	are immediate from the scaling and microscopic subadditivity.
\end{proof}

\begin{remark}
	Convergence of the last wave may be viewed as a sort of bootstrap percolation shape theorem. 
	Sites are initially randomly assigned two thresholds, 1 or 2. A site with threshold $l$
	becomes infected when at least $l$ of its neighbors are infected. Infected sites remain infected. 
	The above shows that if you start off with a large enough cluster of infected sites at the origin, every site will eventually become 
	infected and the speed at which the infection spreads converges. 
	
	For more on the relationship between sandpiles and bootstrap percolation, see Section \ref{sec:criteria} below. 
	Similar shape theorems include \cite{garet2012asymptotic, kesten2008shape,alves2002shape, cox1981some} and 
	especially \cite{willson1978convergence, gravner1993threshold,fey2010limiting}.
\end{remark}

\section{Proof of Theorem \ref{theorem:convergence_of_time_constant}} \label{sec:proof_of_theorem}
Let $\eta$ be drawn from the event of full probability in Proposition \ref{prop:convergence_of_the_last_wave}.
It suffices to show that the last-wave is a good approximation of the original process.

\begin{prop} \label{prop:approximate_subadditive}
	On an event of probability 1, there are constants $C_1(\eta), C_2(\eta)$
	so that for all $x \in \Z^d$, 
	\begin{equation} \label{eq:ineq1}
	T(x) \leq \hat{T}(x) + C_1(\eta)
	\end{equation}
	and
	\begin{equation} \label{eq:ineq2}
	\hat{T}(x) \leq T(x) + C_2(\eta),
	\end{equation}
	where the last-wave arrival time $\hat{T}$ is defined in \eqref{eq:lastwave_fpp}  
	and $T(x) := \min \{ t \geq 0 : v_t(x) > 0 \}$. 
\end{prop}
\begin{proof}
	
	Recall that $v_t$ is the parallel toppling odometer for $\eta + M_{\eta}  \delta_0$ and recall the last-wave odometer $u_t$ from Section \ref{subsec:n_wave}.
	
	We first check \eqref{eq:ineq1}. 
	Let $u'_t$ be the $\mathcal{P}(0)$-frozen parallel toppling odometer with initial conditions $u'_0(x) = 1\{x \in \mathcal{P}(0)\}$ and
	$s_0' = \eta$ 
	where $\mathcal{P}(0)$ is the penultimate-cluster for $0$ defined in \eqref{eq:penultimate_cluster}.
	By Lemma \ref{lemma:last_wave_bound}, for all $t \geq 0$, $u_{t}$ is at most 1 on the inner boundary of the penultimate-cluster, $\partial^{\circ} \mathcal{P}(0)$.  
	Hence, as $\eta \leq (2 d -1)$, 
	by \eqref{eq:frozen_parallel_toppling_recurse} and the definition of the last-wave we have $u'_t |_{\mathcal{P}(0)^c} \geq u_t |_{\mathcal{P}(0)^c}$ for all $t \geq 0$. Also, since $\eta + M_{\eta}  \delta_0$ is not stabilizable, for some $t_0 \geq C_{\eta}$, $v_{t_0} |_{\mathcal{P}(0)} \geq u'_{0} |_{\mathcal{P}(0)}$. 
	Hence by \eqref{eq:odometer_wave_comp}, $v_{t + t_0} |_{\mathcal{P}(0)^c} \geq u'_t |_{\mathcal{P}(0)^c} \geq u_t |_{\mathcal{P}(0)^c}$ for all $t \geq 0$, 
	completing the check of \eqref{eq:ineq1}

	%
	We now verify \eqref{eq:ineq2}. We first consider the special case where only
	one firing at the origin is needed to have an infinite last-wave. 
	
	\subsection*{Step 1: Special case,  \texorpdfstring{$\hat{M}_{\eta}(0) = 1$}{One firing} }
	Denote the reachable sets up to time $t$ for the last wave and exploding sandpile as 
	\begin{equation} \label{eq:reachablesets_definition}
	\begin{aligned}
	\mathcal{R}_t &:= \{ x \in \Z^d : v_t(x) > 0 \} \\
	\hat{\mathcal{R}}_t &:= \{ x \in \Z^d : u_t(x) > 0 \}.
	\end{aligned}
	\end{equation}
	Note, by minimality, if $\hat{M}_{\eta}(0) = 1$, then $\eta(0) + M_{\eta}(0) = 2 d$. 
	This together with $\eta \geq (2 d- 2)$, implies a strong regularity. 
	Specifically, we show by induction that for all $t \geq 0$,
	\begin{equation} \label{eq:induction1}
	\mathcal{R}_t \subseteq \hat{\mathcal{R}}_t 
	\end{equation}
	and 
	\begin{equation} \label{eq:induction2}
	\begin{aligned}
	|v_t(x \pm e_i \pm e_j) - v_t(x)| &\leq 1  \qquad \mbox{for all $x \in \Z^d$ and $e_i \not = e_j$} \\
	|v_t(x \pm e_i) - v_t(x)| &\leq 1   \qquad \mbox{for all $x \in \Z^d$ and $e_i$}
	\end{aligned}
	\end{equation}
	and 
	\begin{equation} \label{eq:induction3}
	v_t(0) \geq \max_{x \in \Z^d} v_t(x).
	\end{equation}
	The base case $t = 0$ is immediate, so suppose \eqref{eq:induction1}, \eqref{eq:induction2}, and \eqref{eq:induction3}
	hold at $t \geq 0$ and we check $(t+1)$.  
	\subsubsection*{Inductive step for \eqref{eq:induction3}}
	\eqref{eq:parallel_toppling_recursion} and \eqref{eq:induction3} at time $t$ imply
	if $x \not = 0$, 
	\begin{align*}
	v_{t+1}(x) &\leq \lfloor \frac{2 d  v_t(0) + \eta(x) }{2 d} \rfloor \\
	&= v_t(0) + \lfloor \frac{\eta(x)}{2 d} \rfloor \\
	&= v_t(0) \\
	&\leq v_{t+1}(0),
	\end{align*}
	as $\eta \leq (2 d - 1)$. 
	
	\subsubsection*{Inductive step for \eqref{eq:induction2}}
	We first check the origin. By \eqref{eq:induction3}, if $v_t(y) = v_t(0)-1$
	for some $y \sim 0$, then $\Delta v_t(0) + 2 d \leq (2 d -1)$. Otherwise, 
	suppose $v_t(e_i + e_j) = v_t(0) - 1$ for some $e_i \not = e_j$ and the origin is unstable at time $t$. Then,
	$v_t(e_i) = v_t(0)$ and $v_t(e_j) = v_t(0)$. This implies, by \eqref{eq:induction2} applied to $e_i$ and $e_j$, that all other neighbors $y \sim (e_i + e_j)$ have a lower bound, $v_t(y) \geq v_t(0) - 1$. 
	Hence, $\Delta v_t(e_i + e_j) \geq 2$, which implies that $(e_i+e_j)$ is unstable at time $t$
	using $\eta \geq (2 d - 2)$. 
	
	Now, take $x \not = 0$ and suppose for sake of contradiction 
	\begin{equation} \label{eq:unstable}
	\Delta v_t(x) + \eta(x) \geq 2 d
	\end{equation}
	but for some adjacent neighbor $(x + e_j)$,  
	\begin{equation} \label{eq:stable_neighbor1}
	v_t(x+e_j) = v_t(x) - 1 \qquad \mbox{and} \qquad  \Delta v_t(x+e_j) + \eta(x+e_j) \leq 2 d -1.
	\end{equation}
	At least two other adjacent neighbors, $y' \sim x$ must satisfy $v_t(y') = v_t(x) + 1$.
	Indeed, otherwise by \eqref{eq:stable_neighbor1} and \eqref{eq:induction2},	 $\Delta v_t(x) \leq 0$, violating
	our assumption \eqref{eq:unstable}. However, one of those neighbors must be 
	$y' = x \pm e_i$ for some $e_i \not= e_j$. This contradicts
	\eqref{eq:induction2} at time $t$ since 
	\[
	v_t(x + e_i) = v_t(x + e_i + (-e_i + e_j)) + 2.
	\]
	
	Next, take a diagonal neighbor,  $(x + e_i + e_j)$ for $i \not = j$, and suppose for sake of contradiction \eqref{eq:unstable}
	but
	\begin{equation} \label{eq:stable_neighbor2}
	v_t(x+e_i+e_j) = v_t(x) - 1 \qquad \mbox{and} \qquad  \Delta v_t(x+e_i+e_j) + \eta(x+e_i+e_j) \leq (2 d -1).
	\end{equation}
	By \eqref{eq:unstable} there must be at least one adjacent neighbor $y \sim x$ with $v_t(y) = v_t(x) + 1$.
	This neighbor cannot be $(x+e_i)$ or $(x+e_j)$ as it would contradict \eqref{eq:induction2} for $(x+e_i+e_j)$.
	Possibly $y  = (x \pm e_{i'})$ for $i' \not \in \{ i,j\}$,  $y = x - e_i$, or $y = x - e_j$.
	In these cases, 
	\begin{equation} \label{eq:large_neighbors}
	v_t(x + e_i) = v_t(x+e_j) = v_t(x) = v_t(x+e_i+e_j) + 1.
	\end{equation}
	Indeed, if not, then, say, $v_t(x+e_i) = v_t(x) - 1$, and so there must be an additional neighbor, 
	$y' \sim x$, $y' \not = y$, with $v_t(y') = v_t(x) + 1$.  But, either $y'$ or $y$ is diagonal to $(x + e_i)$, 
	which contradicts \eqref{eq:induction2}.

	Assuming \eqref{eq:large_neighbors}, the same argument implies $v_t(y'') \geq v_t(x+e_i+e_j)$ for all $y'' \sim (x+e_i+e_j)$.
	Indeed, such a $y''$ with $v_t(y'') = v_t(x + e_i + e_j) - 1$ would be diagonal 
	to either $x + e_i$ or $x + e_j$. 
	This together with \eqref{eq:large_neighbors} shows $\Delta v_t(x + e_i + e_j) \geq 2$. 
	which contradicts \eqref{eq:stable_neighbor2}.

	\subsubsection*{Inductive step for \eqref{eq:induction1}}
	It suffices to check this for $x \not = 0$ as $u_t(0) = 1$. 
	Suppose for sake of contradiction there is some site $x$ with 
	\begin{equation} \label{eq:explode_unstable}
	\Delta v_t(x) + \eta(x) \geq 2 d
	\end{equation}
	but
	\begin{equation} \label{eq:wave_stable}
	\Delta u_t(x) + \eta(x) \leq (2 d -1)
	\end{equation}
	and $u_t(x) = 0$. By \eqref{eq:induction1}, $u_t(x) = v_t(x) = 0$ and hence 
	\begin{equation} \label{eq:neighbor_bound}
	v_t(y) \leq 1 \qquad \mbox{ for all $y \sim x$},
	\end{equation}
	by \eqref{eq:induction2}. However, \eqref{eq:induction1} and \eqref{eq:neighbor_bound} imply that
	$\Delta v_t(x) + \eta(x) \leq \Delta u_t(x) + \eta(x) \leq (2 d -1)$, contradicting \eqref{eq:explode_unstable}.

	\subsection*{Step 2: General case}
	In the general case, we introduce a pair of approximations to which we can apply the arguments of the special case. Let $\tilde{v}$ be the terminal (unfrozen) odometer for $\eta + (M_{\eta}(0)-1)  \delta_0$ and let
	\begin{equation} \label{eq:penultimate_touched}
	\tilde{\mathcal{P}} := \{0\} \cup \{ x \in \Z^d : \mbox{ there is $y \sim x$ with $\tilde{v}(y) > 0 $}\}.
	\end{equation}
	Let $\tilde{w}_t$ be the $\tilde{\mathcal{P}}$-frozen parallel toppling odometer with initial conditions  
	\begin{equation} \label{eq:cluster_frozen}
	\begin{aligned}
	\tilde{w}_0(x) &= 1 \{ x \in \tilde{\mathcal{P}} \} \\
	s_0' &= \eta.
	\end{aligned}
	\end{equation}
	Let $w_t$ be the parallel toppling odometer for $s: \Z^d \to \Z$ where 
	\begin{equation}
	s(x) := 
	\begin{cases}
	\eta(x) &\mbox{ if $x \not \in \tilde{\mathcal{P}}$}  \\
	(2 d - 1) + \delta_0  &\mbox{ otherwise}.
	\end{cases}
	\end{equation}
	Denote the reachable sets for these processes by
	\begin{equation}
	\begin{aligned} 
	\mathcal{R}_t &:= \{ x \in \Z^d : w_t(x) > 0 \} \\
	\tilde{\mathcal{R}}_t &:= \{x \in \Z^d : \tilde{w}_t(x) > 0\}.
	\end{aligned}
	\end{equation}
	
	The same argument as in Step 1 shows that
	\begin{equation} \label{eq:induction_specialcase2}
	\mathcal{R}_t \subseteq \tilde{\mathcal{R}}_t.
	\end{equation}
	
	We claim that we can conclude after proving the following inequalities, 
	\begin{equation} \label{eq:approx1}
	\tilde{w}_t \leq u_{t + c}
	\end{equation}
	\begin{equation} \label{eq:approx2}
	v_{t} \leq \left( w_{t} + \tilde{v}\right).
	\end{equation}
	Indeed, if $v_t(x) > 0$, then \eqref{eq:approx2} implies $\tilde{v}(x) > 0$ or $w_{t}(x) > 0$. In both cases, 
	using either \eqref{eq:induction_specialcase2} or \eqref{eq:cluster_frozen},
	$\tilde{w}_t(x) > 0$ and so by \eqref{eq:approx1} $u_{t+c}(x) > 0$.

	\subsubsection*{Proof of \eqref{eq:approx1}}
	We know that $|\tilde{\mathcal{P}}| = C < \infty$. Hence, at some finite time 
	$u_{c}(\tilde{\mathcal{P}}) \geq 1$. Monotonicity implies $\tilde{w}_t \leq u_{t + c}$.
	
	\subsubsection*{Proof of \eqref{eq:approx2}}
	This is true at $t = 0$, we use \eqref{eq:parallel_toppling_recursion} and induct,
	\begin{align*}
	v_{t+1}(x) &\leq \lfloor \frac{\sum_{y \sim x} v_t(y) + \eta(x) + M_{\eta}  \delta_0}{2 d } \rfloor \\
	&\leq \lfloor \frac{\sum_{y \sim x} (w_t(y) + \tilde{v}(y)) + \eta(x) + M_{\eta}  \delta_0}{2 d } \rfloor \\
	&= \tilde{v}(x) + \lfloor \frac{\sum_{y \sim x} w_t(y) +  \sum_{y \sim x} (\tilde{v}(y) - \tilde{v}(x)) + \eta(x) + M_{\eta}  \delta_0}{2 d } \rfloor \\
	&\leq \tilde{v}(x) + \lfloor \frac{\sum_{y \sim x} w_t(y) +  s(x)}{2 d } \rfloor \\
	&= \tilde{v}(x) + w_{t+1}(x).
	\end{align*}
	The third inequality used $\Delta \tilde{v}(x) + (M_{\eta} - 1)  \delta_0 + \eta \leq (2 d -1)$.

\end{proof}

\begin{remark}
	Some qualitative features of the limit shape are immediate. For example, the origin is an interior point and the limit 
	is invariant with respect to symmetries of the lattice (symmetry may fail in the periodic case introduced in Section \ref{sec:generalization}).	
	Also, a coupling with oriented percolation as in \cite{durrett1981shape, marchand2002strict} can be used to establish a `flat-edge' 
	for $p$ sufficiently close to one in all dimensions. We omit the details since it is routine --- see, for example, the proof of Theorem 1.2 in \cite{alves2002shape} or Theorem 6.3 in \cite{garet2004asymptotic}.
\end{remark}

\section{A generalization}	\label{sec:generalization}																																																																																																																																																																																																																																																																																																																																																																																																																																																																																																																																																																																																																																																																																																																																																																																																																																																																																																																																																																																																																																																																																																																																																																																																																																																																																																																																																																																																																																																																																																																																																																																																																																																																																																																																																																																																																																																																																																																																																																																																																																																																																																																																																																																																																																																																																																																																																																																																																																																																																																																																																																																																																																																																																																																																																																																																																																																																																																																																																																																																																																																																																																																																																																																																																																																																																																																																																																																																																																																																																																																																																																																																																																																																																																																																																																																																																																																																																																																																																																																																																																																																																																																																																																																																																																																																																																																																																																																																																																																																																																																																																																																																																																																																																																																																																																																																																																																																																																																																																																																																																																																																																																																																																																																																																																																																																																																																																																																																																																																																																																																																																																																																																																																																																																																																																																																																																																																																																																																																																																																																																																																																																																																																																																																																																																																																																																																																																																																																																																																																																																																																																																																																										
\subsection{Sufficient hypotheses} \label{sec:sufficient_hypotheses}
We present sufficient hypotheses on $\eta$ under which the arguments above go through seamlessly. 
Fix $\eta_{\min} \in \Z$ and let $\Omega$ denote the set of all bounded functions $\eta: \Z^d \to \Z$, $\eta_{\min} \leq \eta \leq (2 d- 1)$.
Endow $\Omega$ with the $\sigma$-algebra $\mathcal{F}$ generated by $\{ \eta \to \eta(x) : x \in \Z^d\}$. Denote the action of integer translation by $T: \Z^d \times \Omega \to \Omega$, 
\[
T(y, \eta)(z) = (T_y \eta)(z) := \eta(y + z),
\]
and extend this to $\mathcal{F}$ by defining  $T_y E := \{ T_y \eta : \eta \in E\}$.
Let $\mathcal{L} \subseteq \Z^d$ be a {\it sublattice}, a finite index subgroup of $\Z^d$.
Let $\mathbf{P}$ be a {\it stationary} and {\it ergodic} probability measure on $(\Omega, \mathcal{F})$
with respect to $\mathcal{L}$,
\begin{equation}
\mbox{Stationary: for all $E \in \mathcal{F}, y \in \mathcal{L}$: } \mathbf{P}(T_y E) = \mathbf{P}(E),
\end{equation}
\begin{equation}
\mbox{Ergodic:  $E = \bigcap_{y \in \mathcal{L}} T_y E$ implies that $\mathbf{P}(E) \in \{0,1\}$ }.
\end{equation}
We refer to the probability measure $\mathbf{P}$ as {\it explosive} if $\mathbf{P}(\eta \mbox{ is explosive}) = 1$.

Stationarity and ergodicity are the weakest hypotheses under which a convergence result is proved --- straightforward 
counterexamples can be constructed.
However, we do not expect exploding sandpiles to have a limit shape without an additional independence hypothesis. At the very least, 
our proof will not work, as domination by a coarsened product measure was essential. Our first hypothesis is hence a quantification of ergodicity. 

\begin{hyp}[Finite range of dependence] \label{hyp:finite_range_of_dependence}
	There exists a constant $K < \infty$ so that for all $x,y \in \Z^d$, $\eta(x)$ and $\eta(y)$ are independent if $|x-y| > K$. 
\end{hyp}

Next, fix a finite (rectangular) box with side length $k > 0$,  $\mathcal{B}_{k} := \{ x \in \Z^d : 1 \leq x_i \leq k_i \}$.
The $2 d$ external faces of $\mathcal{B}_k$ are 
\[
\begin{aligned}
\mathcal{F}_i &:= \{ x \in \bar{\mathcal{B}}_k : x_i = 0\}, \\
\mathcal{F}_{d+i} &:= \{x \in \bar{\mathcal{B}}_k : x_i = k_i+1\}.
\end{aligned}
\]
Take a face, $\mathcal{F}_i$, and let $w_t: \bar{\mathcal{B}}_k \to \N$ be the sequence of $\mathcal{B}_k^c$-frozen parallel toppling odometers with initial conditions  $w_0 = 1\{ x \in \mathcal{F}_i\}$ and $s_0' = \eta$. We say that $\mathcal{B}_k(\eta)$ can be {\it crossed} in direction $i$ if $w_{t}(x) \geq 1\{x \in \mathcal{B}_k \}$ for $t \geq |\mathcal{B}_k|$.

\begin{hyp}[Box-crossing] \label{hyp:box_crossing}
	For each $\delta > 0$, there is $k$ so that 
	\[
	\min_{j \in \Z^d} \mathbf{P}(\mathcal{B}^{(j)}_k(\eta) \mbox{ can be crossed in each direction}) > 1 - \delta, 
	\]
	where 
	\[
	\bigcup_{j \in \Z^d} \mathcal{B}^{(j)}_{k} := \bigcup_{j \in \Z^d} (\mathcal{B}_k + j  k) = \Z^d
	\]
	is a tiling of the lattice by $\mathcal{B}_k$.
	
\end{hyp}
If $\eta$ were recurrent, Hypothesis \ref{hyp:box_crossing} would imply $\eta$ explodes with probability 1. In particular,  no holes would develop in the support of the odometer. (If unfamiliar, see Section \ref{sec:criteria}
below for the definition of recurrence, although this is not used here.) Our next hypothesis 
ensures this and more: any sufficiently long path of topplings fills its bounding rectangle.

\begin{hyp}[Path-filling] \label{hyp:path_filling}
	There exists a constant $\gamma > 0$ so that on an event of probability 1, for all $n \geq n_0$, and every path of
	distinct points, $[-n,n]^d \supset L_m := \{z_1, \ldots, z_m\}$, $z_{i+1} \sim z_i$, of length $m \geq  (\log n)^{\gamma}$ the following holds. 
	The $L_m$-frozen parallel toppling odometer with initial conditions $w_0 = 1\{ x \in L_m\}$ and $s_0' = \eta$
	quickly exceeds 1 on the bounding rectangle of $L_m$:
	\[
	w_t \geq 1 \{ \br(L_m) \}  \qquad \mbox{ for $t \geq m^d$.}  
	\]
\end{hyp}

In order for $\eta$ to have a limit shape in dimensions $d \geq 3$, we need to strengthen 
Hypothesis \ref{hyp:box_crossing}. The next assumption prevents low-dimensional tendrils from burrowing through good cubes (for a counterexample in three-dimensions take a large cube filled with 4
and connect each of the faces with disjoint tunnels of 5). For a point $z \in \bar{\mathcal{B}}_k$ and direction $1 \leq i \leq d$, consider, as before, a line passing from one side of the box to the other 
\begin{equation}
\mathcal{L}_k^{(i,z)} := \bigcup_{j=1,\ldots,k_i} (z_1, \ldots, z_{i-1}, j, z_{i+1}, \ldots, z_d).
\end{equation}
We say $\mathcal{B}_k(\eta)$ is {\it strongly box-crossing} if, for all $1 \leq i \leq d$ and $z \in \bar{\mathcal{B}}_k$, $w_{|k|} \geq 1\{x \in \mathcal{B}_k \}$,
where $w_t$ is the odometer for $\{\mathcal{B}_k^c \cup \mathcal{L}_k^{(i,z)}\}$-frozen parallel toppling 
with initial conditions $w_0(x) = 1\{x \in \mathcal{L}_k^{(i,z)}\}$, $s_0' = \eta$.

\begin{hyp}[Strongly box-crossing] \label{hyp:box_filling}
	For each $\delta > 0$, there is $k$ so that, using the same notation as Hypothesis \ref{hyp:box_crossing}, 
	\[
	\min_{j \in \Z^d} \mathbf{P}(\mathcal{B}^{(j)}_k(\eta) \mbox{ is strongly box-crossing}) > 1 - \delta.
	\]
\end{hyp}

We now have made enough assumptions to prove convergence of the last-wave as in Section \ref{sec:last_wave}.

\begin{prop}[Convergence of the last-wave] \label{prop:convergence_of_the_generalized_last_wave}
	Under Hypotheses \ref{hyp:finite_range_of_dependence}, \ref{hyp:path_filling}, and \ref{hyp:box_filling},
	on an event of probability 1, $\eta$ is explosive and the rescaled {\it last-wave} arrival times, $n^{-1} \hat{T}_{\eta}([nx]) := n^{-1} \min \{ t > 0 : u_t([nx]) > 0\}$ converge locally uniformly to $\mathcal{N}_{\eta}$, a deterministic, continuous, convex, one-homogeneous function on $\R^d$
	depending only on the law of $\eta$.
\end{prop}
\begin{proof}
	For $\delta > 0$ small, pick side length $k$ from Hypotheses \ref{hyp:box_filling}. 
	For $j \in \Z^d$, let 
	\[
	X_j := 1\{ \mathcal{B}^{(j)}_k  \mbox{ is strongly box-crossing} \}. 
	\]
	By Theorem 0.0 in \cite{liggett1997domination}, $\{X_j\}_{j \in \Z^d}$ stochastically dominates a sequence 
	of Bernoulli independent random variables $\{Y_j\}_{j \in \Z^d}$ with $P(Y_j = 1) \geq (1 - \pi(\delta))$ for $\pi:[0,1]\to[0,1]$
	satisfying $\pi(\delta) \to 0$ as $\delta \to 0$. Therefore, for $\delta > 0 $ sufficiently small, on an event of probability 1, $\{X_j\}_{j \in \Z^d}$ contains an infinite supercritical percolation cluster $\mathcal{C}_{\infty}$. 
	
	The rest of the argument follows almost exactly the proof in Section \ref{sec:last_wave}. 
	The only minor change is in the proof of Lemma \ref{lemma:square_lower_bound}. We use $\eta \geq \eta_{\min}$
	rather than $\eta \geq (2 d - 2)$ and invoke Theorem 4.1 in \cite{levine2009strong} to get that the support of a $(C n^d)$-wave
	contains $[-n,n]^d$ (where the constant $C$ is larger than before).
\end{proof}

If additionally $\eta \geq (2 d - 2)$, then the argument given in Section \ref{sec:proof_of_theorem} implies that the exploding
sandpile is close to the last-wave and hence converges. Simulations indicate $\eta \geq (2d - 2)$ is not necessary, however, we have not found an alternative condition and are forced to assume this: 

\begin{hyp}[Wave-approximation] \label{hyp:wave_approximation}
	Suppose $\eta$ is explosive, let $u_t$ denote the last-wave, $v_t$ the parallel toppling odometer for $\eta + M_{\eta}  \delta_0$,
	and $\hat{T}_{\eta}$, $T_{\eta}$ the respective arrival times. On an event of probability 1, 
	\[
	\sup_{x \in [-n,n]^d} |\hat{T}_{\eta}(x)- T_{\eta}(x)| = o(n).
	\]
\end{hyp}

\begin{theorem}[Convergence of the exploding sandpile]
	Under the assumptions in Proposition \ref{prop:convergence_of_the_generalized_last_wave} and Hypothesis \ref{hyp:wave_approximation},
	on an event of probability 1, $\eta$ is explosive and the rescaled arrival times, $n^{-1} T_{\eta}([nx]) := n^{-1} \min \{ t > 0 : v_t([nx]) > 0\}$ converge locally uniformly to $\mathcal{N}_{\eta}$, a deterministic, continuous, convex, one-homogeneous function on $\R^d$
	depending only on the law of $\eta$.
\end{theorem}

\begin{proof}
	Immediate from Hypothesis \ref{hyp:wave_approximation} and Proposition \ref{prop:convergence_of_the_generalized_last_wave}.
\end{proof}

\begin{remark} 
	Box-crossing with probability 1 implies $\eta$ is recurrent (see Section \ref{sec:criteria} if unfamiliar).  However, not every 
	recurrent sandpile is explosive  --- take $\eta = (2 d - 2)$ and use \cite{fey2010growth} --- and not every exploding sandpile has a recurrent initial condition ---  see Section \ref{sec:counterexample}. 
\end{remark}

\subsection{Examples satisfying the hypotheses}

The simplest way to ensure Hypotheses \ref{hyp:path_filling} and \ref{hyp:wave_approximation}
is to take $\eta \geq (2 d - 2)$. A random background can be built which satisfies the 
rest of the hypotheses using a {\it Bernoulli cloud}, see Figure \ref{fig:bernoulli_clouds}. Take $p > 0$, fix a finite set of points, $\mathcal{S} \subset \Z^d$ (say a triangle, circle, or a line), and independently sample a uniform random variable at each site on the lattice, $\{U_j\}_{j \in \Z^d}$. 
Then, let 
\begin{equation} \label{eq:bernoulli_cloud}
\eta(x)  :=  
\begin{cases} (2 d- 1)  & \mbox{ if there exists $j \in \Z^d$ such that $U_j < p$ and $x \in \{ \mathcal{S} + j \}$} \\
(2d - 2) & \mbox{ otherwise.}
\end{cases}
\end{equation}
Hypothesis \ref{hyp:finite_range_of_dependence} is satisfied as $|\mathcal{S}| < \infty$
and Hypothesis \ref{hyp:box_filling} as $p > 0$. 

Another family of examples is the {\it random checkerboard}. Fix a box $\mathcal{B}$ which tiles the lattice, $\Z^d = \bigcup_{j \in \Z^d} \mathcal{B}_j$.  Take functions $\zeta_1, \ldots, \zeta_m$, defined on the box,  $\zeta_i: \mathcal{B} \to \{ (2 d - 2), (2 d - 1)\}$. Suppose further that $\mathcal{B}(\zeta_1)$ contains at least one site 
with $(2d-1)$ chips along every straight line.  Let $\{Y_j\}_{j \in \Z^d}$ be a field of i.i.d. random variables, where $P(Y_j = i) = p_i$, for $i = 1, \ldots, m$. 
Then, let 
\begin{equation} \label{eq:random_checkerboard}
\eta(x) := \zeta_i(z(x)) \qquad \mbox{ if $Y_{j(x)} = i$}, 
\end{equation}
where $z(x)  \in \mathcal{B}$ is the position of $x$ in its tiled box, $\mathcal{B}_{j(x)}$. 
Finite range
of dependence is immediate by construction. If we further assume $p_1 > 0$, then Hypothesis \ref{hyp:box_filling} is satisfied by the assumption on $\mathcal{B}(\zeta_1)$.

The random checkerboard includes the degenerate case $p_1 = 1$, where $\eta$ is a periodic copy of $\zeta_1$. 
See Table \ref{tab:periodic_exploding_shapes} for pictures. 
In this case, if $\eta \geq (2 d- 2)$ but is not box-crossing, then the background is not explosive by
Theorem 4.2 in \cite{fey2010growth}. However, it is possible to build random (and periodic) checkerboard, exploding sandpiles with $\eta \not \geq (2 d - 2)$. One could then proceed in an adhoc 
manner to check the hypotheses.  However, we have not found a general recipe in this case. The counterexample in Section \ref{sec:counterexample} uses $\eta \not \geq (2d - 2)$.  

\begin{figure}
	\includegraphics[width=0.2\textwidth]{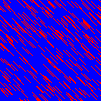}
	\includegraphics[width=0.2\textwidth]{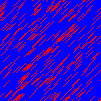}
	\includegraphics[width=0.2\textwidth]{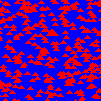}
	\includegraphics[width=0.2\textwidth]{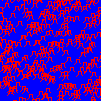} \\
	\includegraphics[width=0.2\textwidth]{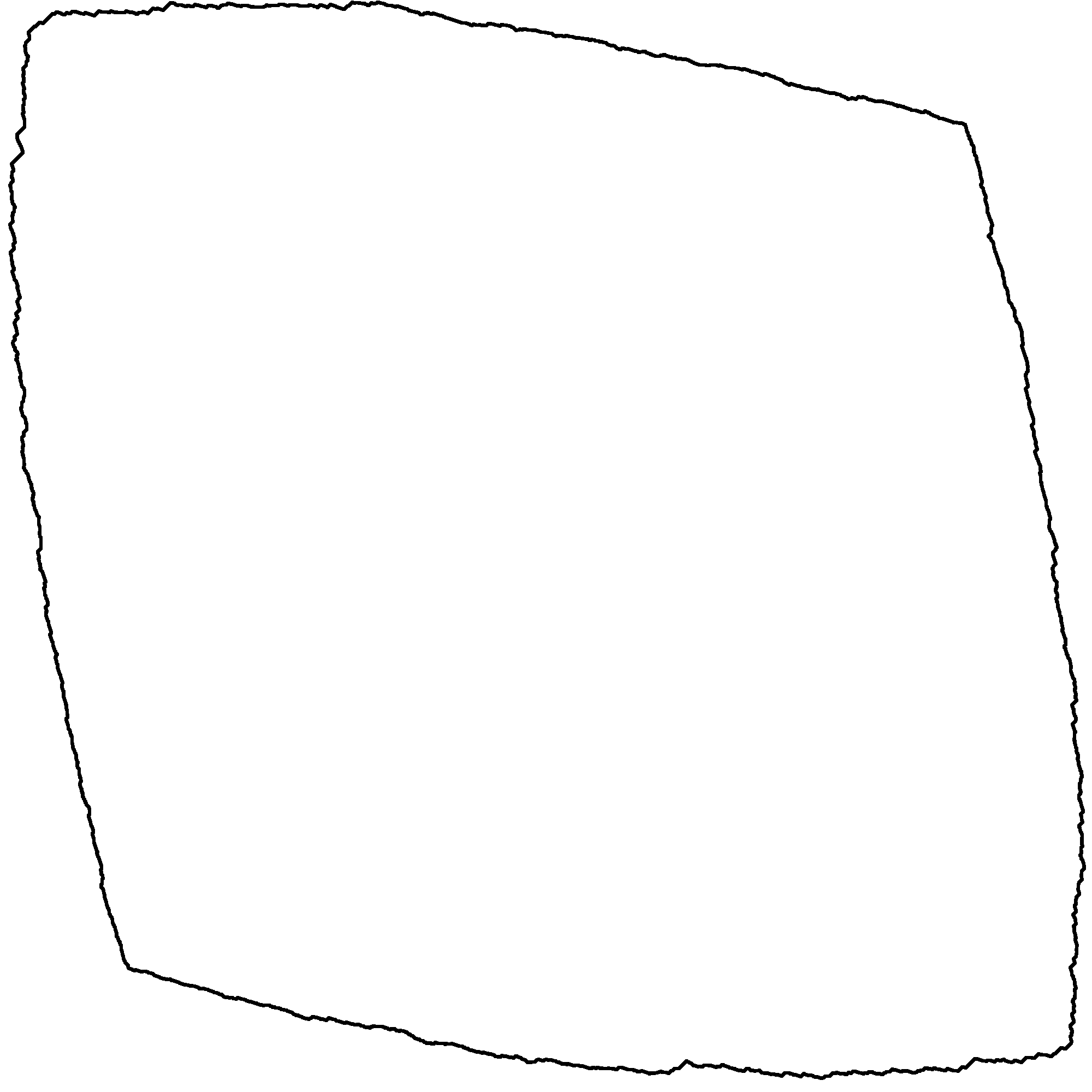}
	\includegraphics[width=0.2\textwidth]{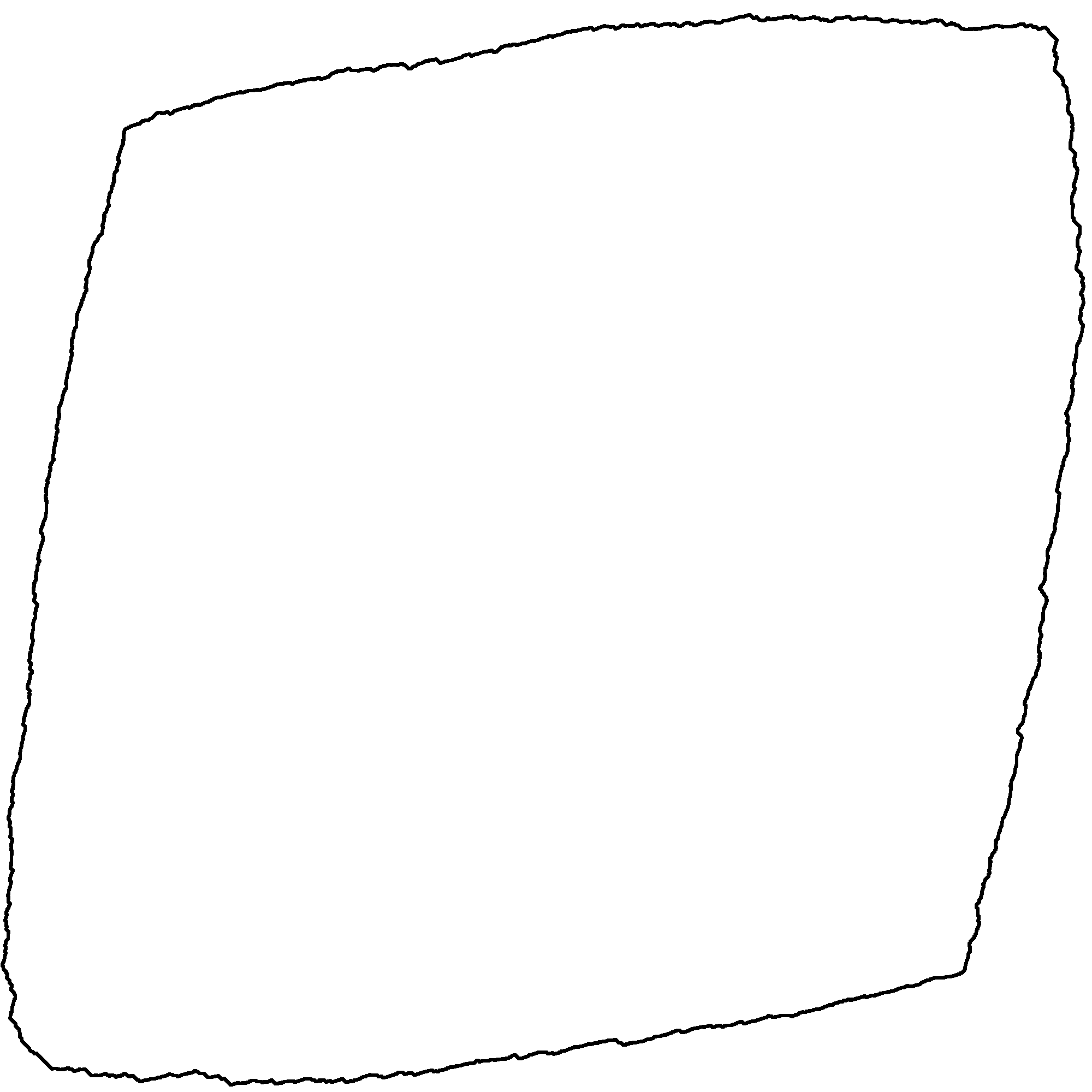}
	\includegraphics[width=0.2\textwidth]{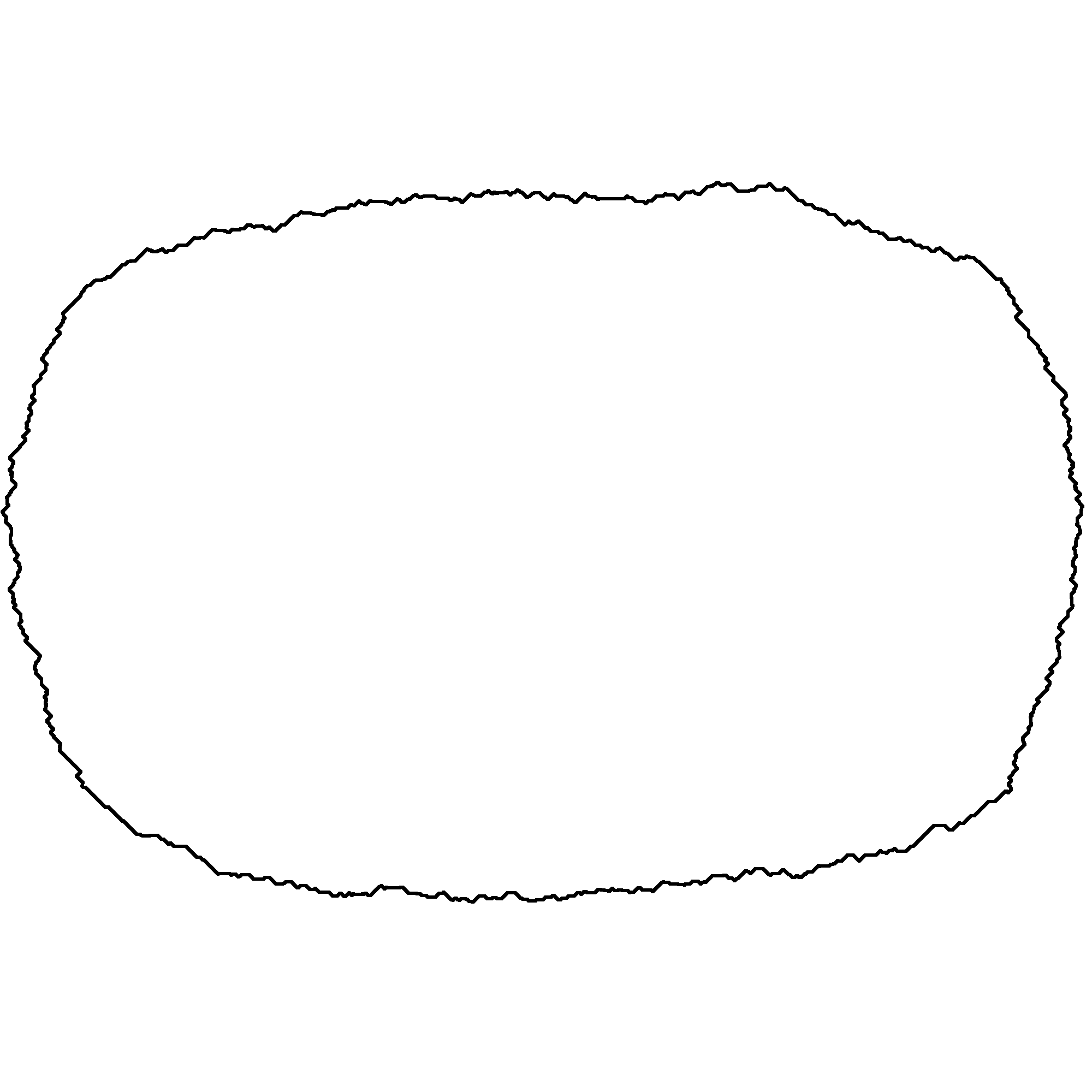}
	\includegraphics[width=0.2\textwidth]{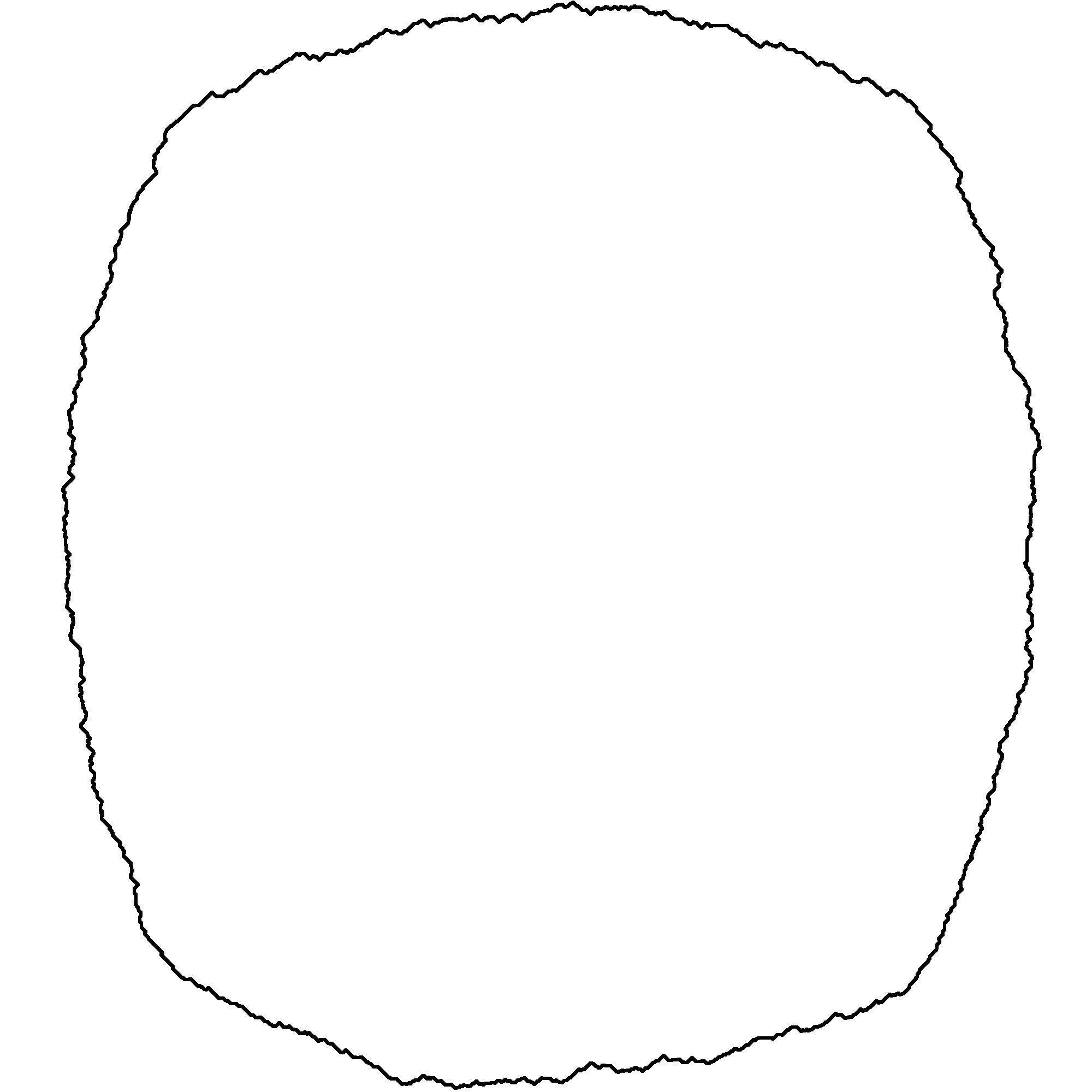}    
	\caption{Initial random backgrounds and computed limit shapes. The random backgrounds
		are built from Bernoulli clouds of the indicated point sets. Blue is 2 chips and red is 3 chips. }
	\label{fig:bernoulli_clouds}
\end{figure} 
\begin{table}
	\begin{tabular}{ c |  c |   c | c | c } \hline
		$\begin{bmatrix} 3 \end{bmatrix}$  & 
		$\begin{bmatrix} 2 & 3 \\
		3 & 2 
		\end{bmatrix}$   &
		$\begin{bmatrix}
		3& 2 & 3 \\ 
		2 & 3 & 3 \\
		3 & 2 & 3 
		\end{bmatrix}$ &
		$\begin{bmatrix}
		2 & 2 & 3  &  3 \\
		3 &  2 &  3 &  2 \\
		2 & 3 & 2 & 2 \\
		2 & 2 &  3 & 2 
		\end{bmatrix}$ &
		$\begin{bmatrix}
		3 & 2 & 3 & 3 & 3 \\
		2 & 2 & 3 & 3 & 3 \\
		2 &  2 & 3 & 2 & 2 \\
		3 & 3 & 2 & 2 & 2 \\
		3 & 3 & 3 & 2 & 3 
		\end{bmatrix}$ 
		\\ \hline
		\includegraphics[width=0.15\textwidth]{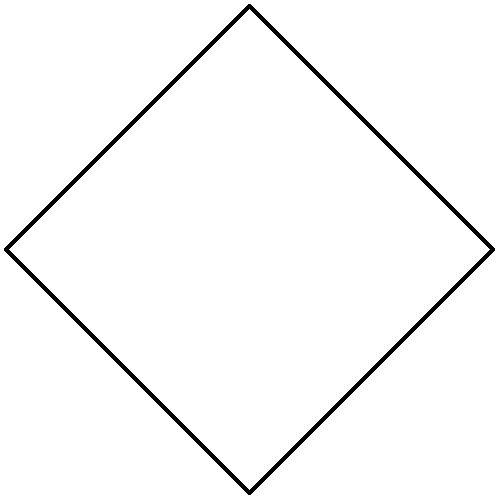} & 	
		\includegraphics[width=0.14\textwidth]{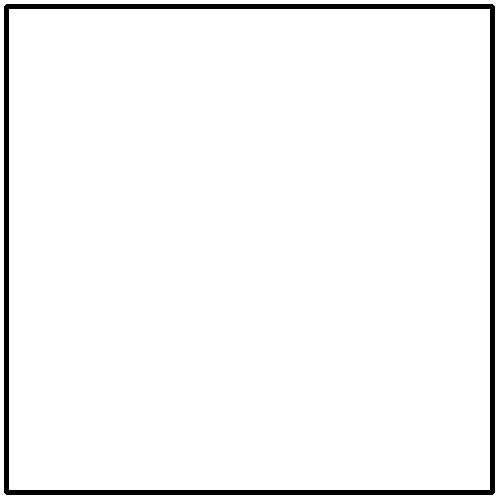} &
		\includegraphics[width=0.15\textwidth]{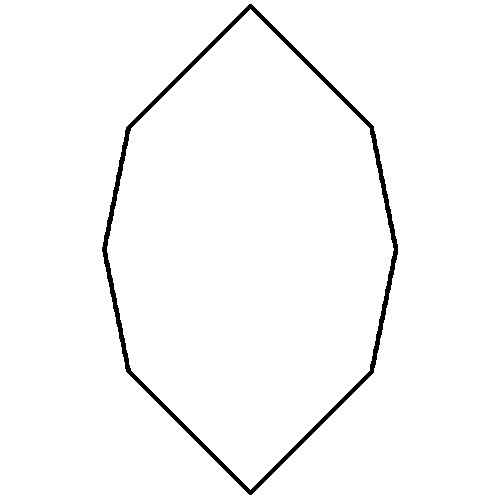} &
		\includegraphics[width=0.15\textwidth]{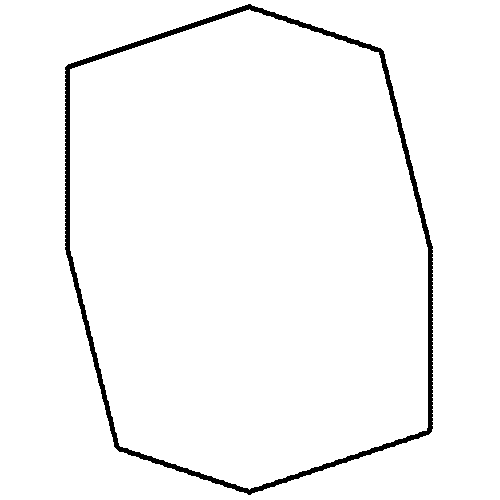} &
		\includegraphics[width=0.15\textwidth]{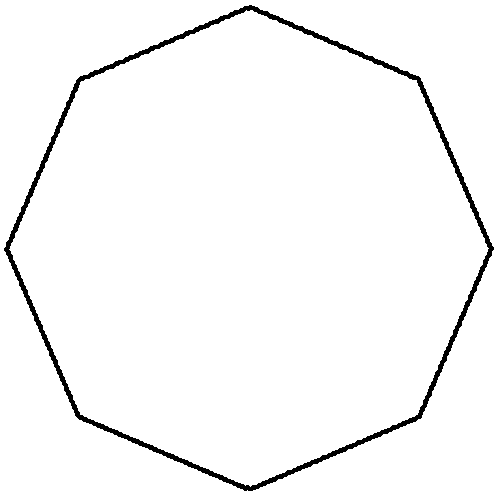} 
		\\ \hline
	\end{tabular}
	\caption{Computed limited shapes of periodic, checkerboard backgrounds of the indicated box.}	
	\label{tab:periodic_exploding_shapes}
\end{table}

\section{Failure of convergence} \label{sec:counterexample}
In this section we construct a family of exploding sandpiles which fail to have a limit shape. 
As the construction indicates, the counterexample is stable: it can be random or periodic. 

\begin{theorem} \label{theorem:counterexample}
	For each dimension $d \geq 2$, there are explosive backgrounds $\eta \not \geq (2d - 2)$ on $\Z^d$ which fail to have a limit shape;
	the first arrival times $\mathcal{T}(n) := \min\{ t > 0 : v_t(n e_1) > 0 \}$ do not converge,
	\begin{equation} \label{eq:counterexample_upper}
	\limsup_{n \to \infty} n^{-1} \mathcal{T}(n) \geq 3/2
	\end{equation}
	and 
	\begin{equation} \label{eq:counterexample_lower}
	\liminf_{n \to \infty} n^{-1} \mathcal{T}(n) = 1.
	\end{equation}
\end{theorem}

\begin{figure}
	\includegraphics[width=0.4\textwidth]{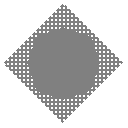}
	\includegraphics[width=0.4\textwidth]{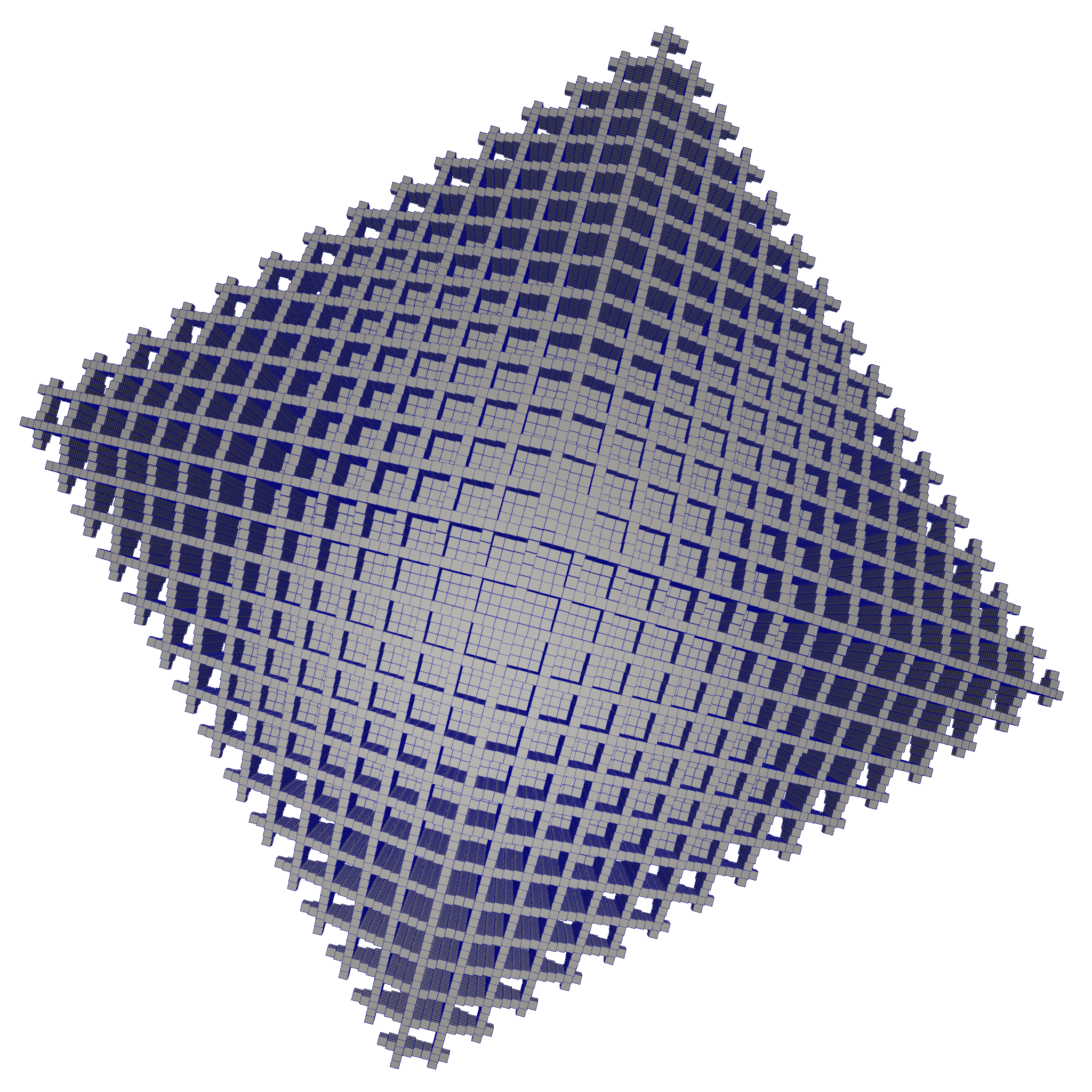}
	\caption{The counterexample from Theorem \ref{theorem:counterexample} for $d = 2$ and $d = 3$.}
	\label{fig:counterexample}
\end{figure}

We explicitly demonstrate a family of checkerboard backgrounds which are explosive but do not have a limit shape. Our counterexample is essentially a two-dimensional one. After constructing it in two dimensions, we embed it into higher-dimensions and show failure of convergence by comparison with the two-dimensional counterexample.

\subsection{Proof of Theorem \ref{theorem:counterexample} for \texorpdfstring{$d=2$}{d=2}} \label{sec:counterexample_2D}
We use the notation of Section \ref{sec:generalization}.  Let $\mathcal{B} := \{ x \in \Z^2 : 0 \leq x \leq 3\}$
denote a box of side length 4 and take $\zeta_1,\zeta_2: \mathcal{B} \to \{1,2, 3\}$ as
\[
\zeta_1 := 
\begin{bmatrix}
1 & 3 & 3 & 1 \\
3 & 3 & 3 & 3 \\
3 & 3 & 3 & 3  \\
1 & 3 & 3 & 1
\end{bmatrix}
\qquad 
\zeta_2 := 
\begin{bmatrix}
1 & 3 & 3 & 1 \\
3 & 2 & 3 & 3 \\
3 & 3 & 3 & 3  \\
1 & 3 & 3 & 1
\end{bmatrix}, 
\]
where the lower-left corner of the box is $(0,0)$ and left-to-right and down-to-up are increasing coordinates.
Let $\eta$ be an arbitrary tiling of $\zeta_1, \zeta_2$; for example, $\eta$ could be a sample from the 
random checkerboard measure. Fix coordinates so that 
\[
\eta(x_1, x_2) = \zeta_i(x_1 \mbox{ mod }  4, x_2  \mbox{ mod }  4).
\]
Let $v_t$ be the sequence of parallel toppling odometers for $s_0 = \eta + 3  \delta_0$. We first verify \eqref{eq:counterexample_lower}.

\subsubsection*{Step 1: Proof of \eqref{eq:counterexample_lower}}
We show for all $n \geq 0$
\begin{equation}
(4 n + 1) \leq \mathcal{T}(4 n + 1) \leq 4 n + 4.
\end{equation} 
By inspection, $v_1(0) = 1$ and $v_2(e_1) = v_2(e_2) = 1$. Now, take $n \geq 1$
and observe that there is a line of 3s connecting $e_2$ to $(e_2 + (4 n + 1)e_1)$.
Thus, $v_{4 n+3}(e_2 + (4 n + 1)e_1) = 1$ and $v_{4 n+4}( (4 n + 1)e_1) = 1$.
The lower bound is immediate from $\eta \leq (2d - 1)$ --- a site can fire only if a neighbor has fired previously.

\subsubsection*{Step 2: $\eta$ is explosive}
By Theorem 2.8 in \cite{fey2009stabilizability}, 
it suffices to construct a toppling procedure
which transforms  $\eta$ into a configuration $\eta'$ which is not stabilizable
in such a way that only sites with at least $2 d$ chips are toppled. 
Start by toppling the origin, then every 3, then every 2, then the $2 \times 2$ box of 1s 
containing the origin,
\[
\eta' := \eta + \Delta \left( 1\{\eta \geq 2\} + \delta_0 + \delta_{-e_1} + \delta_{-e_2} + \delta_{-e_1-e_2} \right).
\]
(It may be checked that this is indeed a legal toppling procedure in the sense of \cite{fey2009stabilizability}.)
The resulting configuration is (away from the origin) a tiling of $\zeta_1',\zeta_2' : \mathcal{B} \to \{2, 3\}$, 
\begin{equation} \label{eq:modified_checkerboard}
\zeta_1' := 
\begin{bmatrix}
3 & 2 & 2 & 3 \\
2 & 3 & 3 & 2 \\
2 & 3 & 3 & 2  \\
3 & 2 & 2 & 3
\end{bmatrix}
\qquad 
\zeta_2' := 
\begin{bmatrix}
3 & 2 & 2 & 3 \\
2 & 2 & 3 & 2 \\
2 & 3 & 3 & 2  \\
3 & 2 & 2 & 3
\end{bmatrix}.
\end{equation}
\begin{remark}
	The reason why convergence fails for this counterexample is that the limit 
	shape of the explosive background $\eta'$ is not a diamond. See Figure \ref{fig:counterexample}.
	When $d=2$ the limit shape is a regular octagon with boundary $\max(|x-y/3|, |x+y/3|, |x/3 - y|, |x/3 + y|)$, but we will not prove this. 
\end{remark}

Both $\zeta_1'$ and $\zeta_2'$ are box-crossing, so it remains to check 
that we can construct a sequence of firings to the outer face
of a box away from the origin. The box containing
the origin is at least
\[
\eta'|_{[0,3]^2} \geq
\begin{bmatrix}
3 & 2 & 2 & 3 \\
2 & 2 & 3 & 2 \\
3 & 3 & 3 & 2 \\
1 & 3 & 2 & 3
\end{bmatrix}.
\]
From this we see that $3  \delta_0 + \eta'$ is not stabilizable --- in a finite number of steps every site in $[0,3]^2$ will fire. 

\subsubsection*{Step 3: Reductions}
Before proving \eqref{eq:counterexample_upper}, we make several reductions.
We seek to lower bound $\mathcal{T}$, therefore, we are free to add
to $\eta$ as this will only decrease the arrival time. First, we may suppose all of the boxes are $\zeta_1$ rather then $\zeta_2$. 

We then increment the background so as to reduce to a sandpile on a cylinder,  $\mathcal{C} := \{ x \in \Z^2 : x_1 \geq 0,  3 \geq x_2 \geq 0\}$. 
Specifically, let $\zeta: \mathcal{C} \to \{1, 3, 4\}$, 
\[
\zeta := 
\begin{bmatrix}
\begin{matrix}
1 & 3 & 3 & 1 \\
3 & 3 & 3 & 3 \\
3 & 3 & 3 & 3 \\
4 & 3 & 3 & 1 
\end{matrix}
& \zeta_1 & \zeta_1 & \cdots
\end{bmatrix}
\]
with the origin, $(0,0)$, on the bottom left with left-to-right, down-to-up 
increasing.  To 
periodically tile by $\zeta$: set for $x_1 \geq 0$, 
\[
\hat{\eta}(x_1, x_2) :=
\zeta(x_1, x_2 \mbox{ mod } 4)
\]
and for $x_1 < 0$, 
\[
\hat{\eta}(x_1, x_2) := \hat{\eta}(-(x_1+1), x_2).
\]
Note that $\hat{\eta} \geq \eta$. 

The structure of $\hat{\eta}$ allows us to reduce
to a symmetrized Laplacian on the cylinder $\mathcal{C}$ (see for example Lemma 2.3 in \cite{bou2020dynamic}) with reflecting boundaries at $x_1 = 0$: $v_t(-1, x_2) = v_t(0,x_2)$
and torus boundary conditions for $x_2 \in \{0,3\}$: $v_t(x_1, -1) = v_t(x_1,3)$, and $v_t(x_1, 4) = v_t(x_1,0)$.  This defines the {\it symmetrized} Laplacian $\Delta$ and nearest neighbors $y \sim x$ on $\mathcal{C}$. 

Let $\tilde{u}(x): \mathcal{C} \to \{0 ,1\}$ be  $\tilde{u}(x) := 1 \{ \hat{\eta}(x) \geq 3\}$.
Then, 
\[
\Delta \tilde{u} + \zeta
=
\begin{bmatrix}
\begin{matrix}
4 & 2 & 2 & 3 \\
2 & 3 & 3 & 2 \\
3 & 3 & 3 & 2 \\
3 & 3 & 2 & 3 
\end{matrix}
& \zeta_1' & \zeta_1' & \cdots 
\end{bmatrix}
=: \zeta'.
\]
Let $v_t: \mathcal{C} \to \Z^+$ be the symmetrized parallel toppling odometer for $\zeta$ and $w_t:\mathcal{C} \to \Z^+$ the
same for $\hat{\eta}'$ (defined with $\zeta'$ as $\hat{\eta}$ was with $\zeta$). We claim that 
\begin{equation} \label{eq:second_limit_reduction}
\mathcal{T}'(3 + 8n) \leq \mathcal{T}(3 + 8 n), 
\end{equation}
where $\mathcal{T}'(n) := \min\{t > 0 : w_t(n e_1) > 0\}$.
In fact, we claim
\begin{equation} \label{eq:topple_compare_induc}
v_t \leq w_t + \tilde{u}
\end{equation}
for all $t \geq 0$.  This includes \eqref{eq:second_limit_reduction} as $\tilde{u}( (3 + 8 n) e_1) = 0$. 
We observe \eqref{eq:topple_compare_induc} is a consequence of induction: the base case $t = 0$ is automatic and 
the inductive step is, 
\begin{align*}
v_{t+1}(x) &=  \lfloor \frac{\sum_{y \sim x}  v_t(y) + \zeta(x)}{4} \rfloor  \\
&\leq  \lfloor \frac{\sum_{y \sim x}  (w_t(y) + \tilde{u}(y)) + \zeta(x)}{4} \rfloor  \\
&= \tilde{u}(x) +  \lfloor \frac{\sum_{y \sim x}  w_t(y) + \Delta \tilde{u}(x) + \zeta(x)}{4} \rfloor  \\
&= \tilde{u}(x) +  w_{t+1}(x).
\end{align*}

\subsubsection*{Step 4: Proof of \eqref{eq:counterexample_upper}}
We show for all $n \geq 1$,
\begin{equation} \label{eq:second_limit_upper_bound}
\mathcal{T}'(3 + 8 n) \geq 12 n.
\end{equation}
We do so by building a `pulsating front' for $w_t$ in the horizontal direction. (Readers interested in pulsating fronts in periodic media on $\R^d$ may see Section 2.2 of \cite{xin2009introduction}.)

We first reduce to the last-wave for $w_t$, $\hat{w}_t$, with initial conditions
$\hat{w}_0 = \delta_{(0,3)}$ and
$\hat{s}_0 = \zeta' + \Delta \hat{w}_0$. The justification is identical to Step 1 of the proof of Proposition \ref{prop:approximate_subadditive}
and so is omitted.  Using $\hat{w}_t \leq 1$, we make another reduction to initial condition
$\hat{w}_0: \mathcal{C} \to \{0, 1\}$, 
\[
\hat{w}_0 = 
\begin{bmatrix} 
\begin{smallmatrix}
1 & 1 & 1 & 1 \\
1 & 1 & 1 & 0 \\
1 & 1 & 1 & 0 \\
1 & 1 & 1 & 1 
\end{smallmatrix}
& \mathbf{0} & \mathbf{0} & \cdots
\end{bmatrix}.
\]
\setcounter{MaxMatrixCols}{50}
We now show, by manual computation, that the configuration of the odometer at the front, the rightmost $4 \times 4$ box in $\mathcal{C}$ containing a site which has toppled, 
is 12-periodic in time. For notational ease, we denote sites which have toppled by $*$, 
\begin{align*}
\hat{s}_0 &= 
\begin{bmatrix}
\begin{smallmatrix}
* & * & * & * & 4 & 2 & 2 & 3 & 3 & 2 & 2 & 3  & 3 \\
* & * & * & 4 & 2 & 3 & 3 & 2 & 2 & 3 & 3 & 2 & 2 \\
* & * & * & 4 & 2 & 3 & 3 & 2 & 2 & 3 & 3 & 2 & 2 \\ 
* & * & * & * & 4 & 2 & 2 & 3 & 3 & 2 & 2 & 3 & 3
\end{smallmatrix}
& 
\cdots
&
\end{bmatrix}
\qquad
\hat{s}_7 = 
\begin{bmatrix}
\begin{smallmatrix}
* & * & * & * & * & * & * & * & * & 3 & 2 & 3 & 3 \\
* & * & * & * & * & * & * & * & 4 & 3 & 3 & 2 & 2  \\
* & * & * & * & * & * & * & * & 4 & 3 & 3 & 2 & 2 \\
* & * & * & * & * & * & * & * & * & 3 & 2 & 3 & 3
\end{smallmatrix}
& 
\cdots
&
\end{bmatrix} \\
\hat{s}_1 &= 
\begin{bmatrix}
\begin{smallmatrix}
* & * & * & * & * & 3 & 2 & 3 & 3 & 2 & 2 & 3 & 3 \\ 
* & * & * & * & 4 & 3 & 3 & 2 & 2 & 3 & 3 & 2 & 2 \\
* & * & * & * & 4 & 3 & 3 & 2 & 2 & 3 & 3 & 2 & 2 \\
* & * & * & * & * & 3 & 2 & 3 & 3 & 2 & 2 & 3 & 3
\end{smallmatrix}
& 
\cdots
&
\end{bmatrix} 
\qquad \hat{s}_8 =
\begin{bmatrix}
\begin{smallmatrix}
* & * & * & * & * & * & * & * & * & 3 & 2 & 3 & 3 \\ 
* & * & * & * & * & * & * & * & * & 4 & 3 & 2 & 2 \\ 
* & * & * & * & * & * & * & * & * & 4 & 3 & 2 & 2 \\ 
* & * & * & * & * & * & * & * & * & 3 & 2 & 3 & 3
\end{smallmatrix}
& 
\cdots
&
\end{bmatrix} \\
\hat{s}_2 &= 
\begin{bmatrix}
\begin{smallmatrix}
* & * & * & * & * & 3 & 2 & 3 & 3 & 2 & 2 & 3 & 3 \\
* & * & * & * & * & 4 & 3 & 2 & 2 & 3 & 3 & 2 & 2 \\ 
* & * & * & * & * & 4 & 3 & 2 & 2 & 3 & 3 & 2 & 2 \\
* & * & * & * & * & 3 & 2 & 3 & 3 & 2 & 2 & 3 & 3
\end{smallmatrix}
& 
\cdots
&
\end{bmatrix} 
\qquad \hat{s}_9 =
\begin{bmatrix}
\begin{smallmatrix}
* & * & * & * & * & * & * & * & * & 4 & 2 & 3 & 3 \\ 
* & * & * & * & * & * & * & * & * & * & 4 & 2 & 2 \\ 
* & * & * & * & * & * & * & * & * & * & 4 & 2 & 2 \\
* & * & * & * & * & * & * & * & * & 4 & 2 & 3 & 3
\end{smallmatrix}
& 
\cdots
&
\end{bmatrix} \\
\hat{s}_3 &= 
\begin{bmatrix}
\begin{smallmatrix}
* & * & * & * & * & 4 & 2 & 3 & 3 & 2 & 2 & 3 & 3 \\ 
* & * & * & * & * & * & 4 & 2 & 2 & 3 & 3 & 2 & 2 \\
* & * & * & * & * & * & 4 & 2 & 2 & 3 & 3 & 2 & 2 \\
* & * & * & * & * & 4 & 2 & 3 & 3 & 2 & 2 & 3 & 3
\end{smallmatrix}
& 
\cdots
&
\end{bmatrix} 
\qquad \hat{s}_{10} = 
\begin{bmatrix}
\begin{smallmatrix}
* & * & * & * & * & * & * & * & * & * & 4 & 3 & 3 \\ 
* & * & * & * & * & * & * & * & * & * & * & 3 & 2 \\ 
* & * & * & * & * & * & * & * & * & * & * & 3 & 2 \\ 
* & * & * & * & * & * & * & * & * & * & 4 & 3 & 3
\end{smallmatrix}
& 
\cdots
&
\end{bmatrix} \\
\hat{s}_4 &= 
\begin{bmatrix}
\begin{smallmatrix}
* & * & * & * & * & * & 4 & 3 & 3 & 2 & 2 & 3 & 3 \\ 
* & * & * & * & * & * & * & 3 & 2 & 3 & 3 & 2 & 2 \\ 
* & * & * & * & * & * & * & 3 & 2 & 3 & 3 & 2 & 2 \\ 
* & * & * & * & * & * & 4 & 3 & 3 & 2 & 2 & 3 & 3
\end{smallmatrix}
& 
\cdots
&
\end{bmatrix}
\qquad \hat{s}_{11} = 
\begin{bmatrix}
\begin{smallmatrix}
* & * & * & * & * & * & * & * & * & * & * & 4 & 3 \\ 
* & * & * & * & * & * & * & * & * & * & * & 3 & 2 \\
* & * & * & * & * & * & * & * & * & * & * & 3 & 2 \\
* & * & * & * & * & * & * & * & * & * & * & 4 & 3
\end{smallmatrix}
& 
\cdots
&
\end{bmatrix} \\
\hat{s}_5 &= 
\begin{bmatrix}
\begin{smallmatrix}
* & * & * & * & * & * & * & 4 & 3 & 2 & 2 & 3 & 3 \\ 
* & * & * & * & * & * & * & 3 & 2 & 3 & 3 & 2 & 2 \\ 
* & * & * & * & * & * & * & 3 & 2 & 3 & 3 & 2 & 2 \\ 
* & * & * & * & * & * & * & 4 & 3 & 2 & 2 & 3 & 3
\end{smallmatrix}
& 
\cdots
&
\end{bmatrix} 
\qquad \hat{s}_{12} = 
\begin{bmatrix}
\begin{smallmatrix}
* & * & * & * & * & * & * & * & * & * & * & *  & 4 \\
* & * & * & * & * & * & * & * & * & * & * & 4  & 2 \\
* & * & * & * & * & * & * & * & * & * & * & 4 &  2 \\ 
* & * & * & * & * & * & * & * & * & * & * & * & 4
\end{smallmatrix}
& 
\cdots
&
\end{bmatrix}. \\
\hat{s}_6 &= 
\begin{bmatrix}
\begin{smallmatrix}
* & * & * & * & * & * & * & * & 4 & 2 & 2 & 3 & 3 \\ 
* & * & * & * & * & * & * & 4 & 2 & 3 & 3 & 2 & 2 \\ 
* & * & * & * & * & * & * & 4 & 2 & 3 & 3 & 2 & 2 \\ 
* & * & * & * & * & * & * & * & 4 & 2 & 2 & 3 & 3
\end{smallmatrix}
& 
\cdots
&
\end{bmatrix}
\end{align*}

This shows that
\[
\hat{w}_{12} = 
\begin{bmatrix}
\begin{smallmatrix}
1 & 1 & 1 & 1 & 1 & 1 & 1 & 1 & 1 & 1 & 1 & 1 \\ 
1 & 1 & 1 & 1 & 1 & 1 & 1 & 1 & 1 & 1 & 1 & 0 \\ 
1 & 1 & 1 & 1 & 1 & 1 & 1 & 1 & 1 & 1 & 1 & 0 \\ 
1 & 1 & 1 & 1 & 1 & 1 & 1 & 1 & 1 & 1 & 1 & 1
\end{smallmatrix}
& 
\mathbf{0} 
&
\mathbf{0} 
& \cdots
\end{bmatrix},
\]
so the odometer at the front is identical to what it was at the start and the process `resets'. Hence, 
by induction on $n$,
\begin{equation}
\hat{T}(3 + 8 n) = 12 n, 
\end{equation}
where $\hat{T}(n) := \min\{ t > 0 : \hat{w}_t(n e_1) > 0 \}$. 
This implies  \eqref{eq:second_limit_upper_bound}, completing the proof by \eqref{eq:second_limit_reduction}.

\subsection{Proof of Theorem \ref{theorem:counterexample} for \texorpdfstring{$d \geq 3$}{d > 2}} \label{sec:proof_of_generalization}
Let $\eta^{(2D)}: \Z^2 \to \{1, 2, 3\}$ be the (possibly random) background defined in Section \ref{sec:counterexample_2D} and let $d \geq 3$ be given. Our higher-dimensional counterexample $\eta: \Z^d \to \{ 2 d - 3, 2 d -2, 2 d - 1\}$
is built by stacking the two-dimensional one, 
\begin{equation} \label{eq:stacked_tiling1}
\eta(x_1, x_2, \ldots, x_d) = 2  (d- 2) + \eta^{(2D)}(x_1, x_2) \mbox{ for all $x \in \Z^d$}.
\end{equation}
\subsubsection*{Step 1: Proof of \eqref{eq:counterexample_lower}}
The argument is identical to $d = 2$.

\subsubsection*{Step 2: $\eta$ is explosive}

Let ${\eta'}^{(2D)}$ be the tiling of $\zeta_1', \zeta_2'$, defined in \eqref{eq:modified_checkerboard}.
The higher-dimensional analogue, $\eta': \Z^d \to \{ 2 d - 3, 2 d -2, 2 d - 1\}$ is also stacked, 
\begin{equation} \label{eq:stacked_tiling2}
\eta'(x_1, x_2, \ldots, x_d) := {\eta'}^{(2D)}(x_1, x_2) + 2  (d - 2).
\end{equation}
The argument is as before: we construct a toppling procedure that transforms $\eta$ into $\eta'$. Since $3  \delta_0 + \eta' \geq (2 d - 2)$ is box-crossing, it is not stabilizable. 

Topple the origin, all sites with $(2d - 1)$ chips, then all sites with $(2d - 2)$, then the column of $(2 d-3)$ near the origin. Let $\tilde{u}$ denote the odometer for this and $\tilde{u}^{(2D)}$ the two-dimensional version and observe that $\tilde{u}(x_1, x_2, \ldots, x_d) = \tilde{u}^{(2D)}(x_1,x_2)$. This implies,
\begin{align*}
&\sum_{i=1}^{d} (\tilde{u}(x-e_i) + \tilde{u}(x+e_i) - 2 \tilde{u}(x)) + \eta \\
&= \sum_{i=1}^2 (\tilde{u}^{(2D)}(x-e_i) + \tilde{u}^{(2D)}(x+e_i) - 2 \tilde{u}^{(2D)}(x)) + \eta \\
&= \sum_{i=1}^2 (\tilde{u}^{(2D)}(x-e_i) + \tilde{u}^{(2D)}(x+e_i) - 2 \tilde{u}^{(2D)}(x)) + \eta^{(2D)} + 2  (d-2) \\
&= {\eta'}^{(2D)} + 2  (d-2) \\
&= \eta'. 
\end{align*}
We conclude by observing $\eta'$ is box-crossing as $3  \delta_0 + \eta' \geq (2 d - 2)$ and every layer in the box  $\mathbf{B}^{(d)} := \{ x \in \Z^d : 0 \leq x \leq 3 \}$, contains at least one site with $(2 d - 1)$ chips. Indeed, for all $(3, 3, \mathbf{x}_{d-2}) \in \mathbf{B}^{(d)}$, 
\[
\eta'(3,3, \mathbf{x}_{d-2}) = {\eta'}^{(2D)}(3, 3) + 2  (d - 2) = (2 d - 1),
\]
and every layer in ${\eta'}^{(2D)}$ has at least one site with 3 chips.

\subsubsection*{Step 3: Proof of \eqref{eq:counterexample_upper}}
Let $v_t^{(2D)}$ be the parallel toppling odometer for $\eta^{(2D)}$. By \eqref{eq:second_limit_upper_bound}
it suffices to show that
\begin{equation} \label{eq:parabolic_least_action}
v_t(x_1, x_2, \ldots, x_d) \leq v_t^{(2D)}(x_1,x_2). 
\end{equation}
This is a consequence of the parabolic least action principle (Lemma 2.3 in \cite{bou2020dynamic}) but 
we provide a self-contained proof here:

Suppose \eqref{eq:parabolic_least_action} holds at time $t$ and we want to show it holds at $(t+1)$.
Let $x$ be given. If $v_t(x) < v_t^{(2D)}(x_1,x_2)$ we are done. Hence, we may suppose $v_t(x) = v_t^{(2D)}(x_1,x_2)$
and 
\[
\Delta v_t(x) + \eta = \sum_{i=1}^d (v_t(x + e_i) + v_t(x - e_i) - 2 v_t(x)) + \eta \geq 2 d.
\]
By \eqref{eq:parabolic_least_action} at time $t$, this implies 
\[
\sum_{i=1}^2 (v_t^{(2D)}((x_1,x_2) + e_i) + v_t^{(2D)}((x_1,x_2) - e_i) - 2 v_t^{(2D)}(x_1,x_2)) + \eta \geq 2 d,
\]
and by \eqref{eq:stacked_tiling1}, 
\[
\Delta v_t^{(2D)}(x_1,x_2) + \eta^{(2D)}(x_1,x_2) \geq 2 d - 2  (d - 2) = 4,
\]
completing the proof.

\section{A criterion for exploding} \label{sec:criteria}

Let $(\Omega, \mathcal{F}, \mathbf{P})$ be a probability space of sandpile backgrounds defined in Section \ref{sec:generalization}. 
For a finite domain $V \subset \Z^d$, we say  $\eta: V \to \Z$ is {\it recurrent} if the firing of $\partial V$ causes every site in $V$ to eventually topple. 
Specifically, the $V^c$-frozen parallel toppling odometer for initial conditions $w_0 = 1\{ x \in \partial V\}$, $s_0' = \eta$, 
is eventually 1 on $V$: $w_t |_{V} = 1$ for $t \geq |V|$. We say $\eta: \Z^d \to \Z$ is recurrent if its
restriction to $V$ is recurrent for every finite $V \subset \Z^d$.  The measure $\mathbf{P}$ is recurrent if $\mathbf{P}(\mbox{$\eta$ is recurrent}) = 1$. The arguments in Section \ref{sec:generalization} imply the following criterion. 

\begin{prop} \label{prop:explosion_criteria}
	If Hypotheses \ref{hyp:finite_range_of_dependence} and \ref{hyp:box_crossing} 
	are satisfied and $\mathbf{P}$ is recurrent
	then $\mathbf{P}$ is explosive. 
\end{prop}

In this section, we use Proposition \ref{prop:explosion_criteria} to prove Theorem \ref{theorem:explosive}.
Let $\mathbf{P}$ denote a product measure with $\mathbf{P}(\eta \geq d) = 1$ and $\mathbf{P}(\eta(0) = 2 d -1) > 0$. 
We show that $\mathbf{P}$ is recurrent and box-crossing. 
Both arguments require a form of dimensional reduction, which we record in the first subsection.
Also, by monotonicity, we may assume $\eta \to \{d, 2 d - 1\}$.

\subsection{Dimensional reduction}
Let $Q_n^{(d)} := \{ x \in \Z^d : 1 \leq x \leq n\}$ and denote each of the (internal) $2d$ faces of $Q_n^{(d)}$ as 
\begin{equation}\label{eq:faces_definition}
\begin{aligned}
\mathcal{F}_i(Q_n^{(d)}) &:= \{ x \in Q_n^{(d)} : x_i = 1\} \\
\mathcal{F}_{d+i}(Q_n^{(d)}) &:= \{ x \in Q_n^{(d)} : x_i = n\}.
\end{aligned}
\end{equation}
We show that after firing the outer boundary of $\mathcal{F}_i(Q_n^{(d)})$, 
the sandpile dynamics on $\mathcal{F}_i(Q_n^{(d)})$ can be coupled
with $(d-1)$-dimensional sandpile dynamics on $Q_n^{(d-1)}$.

Specifically, fix $i = 1$,  and for each $x \in \mathcal{F}_{1}(Q_n^{(d)})$ write $x = (1, \mathbf{x}_{d-1})$. 
Let the initial $d$-dimensional background be given, $\eta^{(d)}: Q_n^{(d)} \to \Z$. 
We consider the sequences $s_t^{(d-1)}$, $u_t^{(d-1)}$, and $s_t^{(d)}$, $u_t^{(d)}$ of $(Q_n^{(d)})^c$, $(Q_n^{(d-1)})^c$ {\it constrained  frozen toppling processes} with initial conditions,
\begin{equation} \label{eq:constrained_toppling}
\begin{aligned}
u_0^{(d)}(0,\mathbf{x}_{d-1}) &= 1 \\
u_0^{(d-1)}(\mathbf{x}_{d-1}) &= u_0^{(d)}(1, \mathbf{x}_{d-1}) \\
s_0^{(d)}(1, \mathbf{x}_{d-1}) &=  \Delta u_0^{(d)}(1, \mathbf{x}_{d-1}) + \eta^{(d)}(1,\mathbf{x}_{d-1}) \\
s_0^{(d-1)}(\mathbf{x}_{d-1}) &=  \eta^{(d)}(\mathbf{x}_{d-1})-1 := \eta^{(d-1)}(\mathbf{x}_{d-1})
\end{aligned}
\end{equation}
and constraints,
\begin{equation} \label{eq:frozen_constraint}
\begin{aligned}
u_t^{(d)}(x_1,\mathbf{x}_{d-1}) &= 0 \quad \mbox{ for all $x_1 > 1$} \\
u_t^{(d)}(1, \mathbf{x}_{d-1}) & \leq 1 \\
u_t^{(d-1)}(\mathbf{x}_{d-1}) &\leq 1,
\end{aligned}
\end{equation}
for all $t \geq 0$. Specifically, we take initial conditions given by \eqref{eq:constrained_toppling}
and enforce the constraints \eqref{eq:frozen_constraint} by defining the odometer sequences for $t \geq 0$ as
\[
u_{t+1}^{(d)}(x) = 
\begin{cases}
u_t^{(d)}(x) \qquad \mbox{if $x \in (Q_n^{(d)})^c$} \\
u_t^{(d)}(x) \qquad \mbox{if $x_1 > 1$} \\
u_t^{(d)}(x) \qquad \mbox{if $x_1 = 1$ and $u_t^{(d)}(x)  = 1$} \\
u_t^{(d)}(x) + 1\{s_t^{(d)}(x) \geq 2 d\} \qquad \mbox{otherwise}
\end{cases}
\]
and
\[
u_{t+1}^{(d-1)}(\mathbf x_{d-1}) = 
\begin{cases}
u_t^{(d-1)}(\mathbf x_{d-1}) \qquad \mbox{if $x \in (Q_n^{(d-1)})^c$} \\
u_t^{(d-1)}(\mathbf x_{d-1}) \qquad \mbox{if $u_t^{(d-1)}(\mathbf{x}_{d-1}) = 1$} \\
u_t^{(d-1)}(\mathbf x_{d-1}) + 1\{s_t^{(d-1)}(\mathbf x_{d-1}) \geq 2 (d-1)\} \qquad \mbox{otherwise}.
\end{cases}
\]
The sandpile sequences are defined as in the unconstrained case. 
Observe that constrained frozen toppling odometers are dominated by the corresponding unconstrained odometers.

We prove the following by an induction on time. Note that a symmetric result holds for every face. 
\begin{prop} \label{prop:dimensional_reduction}
	For all $t \geq 0$ and $(1, \mathbf{x}_{d-1}) \in \mathcal{F}_1(Q_n^{(d)})$, 
	\begin{equation} \label{eq:odometer_reduction}
	u_t^{(d)}(1,\mathbf{x}_{d-1}) = u_t^{(d-1)}(\mathbf{x}_{d-1})
	\end{equation}
	and
	\begin{equation} \label{eq:sandpile_reduction}
	s_t^{(d)}(1,\mathbf{x}_{d-1}) = s_t^{(d-1)}(\mathbf{x}_{d-1}) + 2 \qquad \mbox{ if $ u_t^{(d)}(1,\mathbf{x}_{d-1}) = 0 $}.
	\end{equation}
\end{prop}

\begin{proof}
	For all $t \geq 0$, if  $u_{t}^{(d)}(1,\mathbf{x}_{d-1}) = 0$ and $u_t^{(d)}(1,\mathbf{y}_{d-1}) = u_t^{(d-1)}(\mathbf{y}_{d-1})$ for all $\mathbf{y}_{d-1}$, 
	\begin{align*}
	&s_{t}^{(d)}(1, \mathbf{x}_{d-1}) \\
	&=  \eta^{(d)}(1, \mathbf{x}_{d-1})  \\
	&+ \left( -2 u_{t}^{(d)}(1, \mathbf{x}_{d-1}) + u_{t}^{(d)}(0, \mathbf{x}_{d-1}) + u_{t}^{(d)}(2, \mathbf{x}_{d-1}) \right) \\
	&+ \sum_{i=2}^d \left( u_{t}^{(d)}((1,\mathbf{x}_{d-1}) + e_i) + u_{t}^{(d)}((1,\mathbf{x}_{d-1})-e_i) - 2 u_{t}^{(d)}(1,\mathbf{x}_{d-1}) \right) \\
	&= \eta^{(d-1)}(\mathbf{x}_{d-1}) + 2 \\
	&+  \sum_{i=1}^{d-1} \left( u_{t}^{(d-1)}(\mathbf{x}_{d-1} + e_i) + u_{t}^{(d-1)}(\mathbf{x}_{d-1}-e_i) - 2 u_{t}^{(d-1)}(\mathbf{x}_{d-1}) \right) \\
	&=s_{t}^{(d-1)}(\mathbf{x}_{d-1})+2.
	\end{align*}
	
	Therefore, we may begin the induction and suppose \eqref{eq:odometer_reduction} and \eqref{eq:sandpile_reduction} hold at time $t$.
	If $s_t^{(d-1)}(\mathbf{x}_{d-1}) \geq 2  (d - 1) = (2 d - 2)$, then 
	$s_t^{(d)}(1,\mathbf{x}_{d-1}) = 2 d$. The other direction is identical, showing $u_{t+1}^{(d)}(1, \mathbf{x}_{d-1}) = u_{t+1}^{(d-1)}(\mathbf{x}_{d-1}) = 1$ in this case. 
	
\end{proof}

\subsection{\texorpdfstring{$\mathbf{P}$}{P} is recurrent}
By monotonicity of recurrence, it suffices to prove the following. 
\begin{prop} \label{prop:recurrence}
	For every $d \geq 1$, $\eta:\Z^d \to \{d\}$ is recurrent.
\end{prop}
\begin{proof}
	By consistency of recurrence, it suffices to show this for domains which are cubes (see for example Remark 3.2.1 in \cite{redig2005mathematical}). Write $Q_n^{(d)}$ for a cube of side length $n$ in $\Z^d$.
	We induct on dimension, then cube side length. The base case for dimension $d = 1$ is immediate. 
	Moreover, the base case $n = 2$ is also immediate for every dimension.
	It remains to check $\eta: Q_n^{(d)} \to \{d\}$ is recurrent given $\eta: Q_{n-2}^{(d)} \to \{d\}$ is recurrent.
	
	We decompose the cube into its faces and an inner cube, 
	\begin{equation} \label{eq:cube_decomp}
	Q_{n}^{(d)} = Q_{n-2}^{(d)} \cup \bigcup_{i=1}^{2d} \mathcal{F}_i(Q_n^{(d)}).
	\end{equation}
	By the inductive hypotheses on $n$, once every external face of $Q_{n-2}^{(d)}$ is toppled, every site in $Q^{(d)}_{n-2}$ 
	eventually fires. Therefore, by \eqref{eq:cube_decomp}, it suffices to check that every site in  
	$\mathcal{F}_i(Q_n^{(d)})$ fires after the boundary of $Q_n^{(d)}$ fires.  This, however, is a consequence of the inductive hypothesis
	on dimension and Proposition \ref{prop:dimensional_reduction}, any site in $\mathcal{F}_i(Q_n^{(d)})$ which topples for the $(d-1)$-dimensional process also topples in $d$-dimensions.
	
\end{proof}
\begin{remark}
	The argument given here is similar to the proof of Lemma 3.1 in \cite{schonmann1992behavior}.
\end{remark}

\subsection{\texorpdfstring{$\mathbf{P}$}{P} is box-crossing}
A coupling between the sandpile and bootstrap percolation has been observed before  \cite{fey2010growth}. 
Bootstrap percolation is a cellular automata on $\Z^d$ with a random initial state and a deterministic update rule. 
Every site $x \in \Z^d$ starts off as {\it infected} independently at random with probability $p$. 
Infected sites remain infected and if an uninfected site contains at least $d$ neighbors which
are infected, it becomes infected.

These dynamics exactly match parallel toppling for a background $\eta':\Z^d \to \{d, 2 d\}$
where sites are constrained to topple at most once.
Infected sites are those which have toppled, and sites with $2 d$ chips start as infected. Indeed, any 
site beginning with $d$ chips topples if and only if it has at least $d$ neighbors which have toppled.

Our proof that $\mathbf{P}$ is box-crossing uses this coupling together with a large deviation result of Schonmann. Borrowing the terminology of Schonmann, we say a cube $Q_n^{(d)} \subset \Z^d$ is $\eta'-${\it internally spanned}
if the $(Q_n^{(d)})^c$-frozen parallel toppling procedure with $u_0 = 0$, $s_0' = \eta'$, and sites constrained to topple at most once, concludes with every site in $Q_n^{(d)}$ toppling.

\begin{prop}[\cite{schonmann1992behavior}] \label{prop:schonnmann_bootstrap}
	Let
	\[
	\eta'(x) := \begin{cases} 2 d &\mbox{with probability $p$} \\
	d &\mbox{otherwise.}
	\end{cases}
	\]
	There are constants $c,C$ depending only on dimension and $p > 0$
	so that 
	\begin{equation}
	P(\mbox{$Q_n^{(d)}$ is $\eta'$-internally spanned}) \geq 1 -  c \exp(-C n).
	\end{equation}
\end{prop}

We use Proposition \ref{prop:schonnmann_bootstrap} to show $\mathbf{P}$ is box-crossing. 

\begin{prop}
	In all dimensions, $\mathbf{P}$ is box-crossing
\end{prop}
\begin{proof}
	The claim is immediate in dimension one.
	Let $(d+1) \geq 2$ be given. For $n \geq 1$, decompose the box into layers
	\[
	Q_n^{(d+1)} = \bigcup_{i=1}^{n} \mathcal{L}_i,
	\]
	where $\mathcal{L}_i := \{ x \in Q_n^{(d+1)} : x = (i, \mathbf{x}_d)\}$.
	The projection of each layer to a $d$-dimensional box is $\mathcal{L}_i^{(d)}$. Let 
	\begin{align*}
	\Omega_i' &:= \{\eta : \mbox{$\mathcal{L}_i^{(d)}$ is $\eta'$-internally spanned where $\eta'(\mathbf{x}_d) := \eta(i, \mathbf{x}_d) - 1$} \}
	\end{align*}
	and let 
	\[
	\Omega'  := \cap_{i=1}^n \Omega_i. 
	\]
	By definition of being internally spanned and Proposition \ref{prop:dimensional_reduction}, if $\eta \in \Omega'$, 
	then $Q_n^{(d+1)}(\eta)$ can be crossed in direction $e_1$. We conclude by symmetry and Proposition \ref{prop:schonnmann_bootstrap}.
	
\end{proof}

\bibliography{exploding_rand.bib}{}
\bibliographystyle{amsalpha}

\end{document}

%% file: defs.tex
\newcommand{\R}{\mathbb{R}}

\newcommand{\Z}{\mathbb{Z}}
\newcommand{\N}{\mathbb{N}}

\newcommand{\br}{\mathbf{br}}  

\newcommand{\eg}{{\it e.g.}}
\newcommand{\ie}{{\it i.e.}}

\newtheorem{theorem}{Theorem}[section]
\newtheorem{prop}{Proposition}[section]

\newtheorem{lemma}[theorem]{Lemma}
\newtheorem{remark}{Remark}
\newtheorem{hyp}{Hypothesis}